\documentclass[10pt, a4paper, reqno]{amsart}

\usepackage{amsthm, amsmath, amssymb, amsfonts, stmaryrd, latexsym, mathrsfs, enumerate}
\usepackage{graphicx}
\usepackage[english]{babel}
\usepackage[all]{xy}
\usepackage{nicefrac}
\usepackage{mparhack}
\usepackage{enumitem}

\theoremstyle{plain}
\newtheorem{lemma}{Lemma}
\newtheorem{theo}{Theorem}
\newtheorem{prop}{Proposition}

\newtheorem{condition}{Condition}

\theoremstyle{definition}

\newcommand{\F}{\mathbb{F}}
\newcommand{\Z}{\mathbb{Z}}
\newcommand{\Q}{\mathbb{Q}}

\newcommand{\kp}{\mathfrak{p}}

\newcommand{\sD}{\mathscr{D}}

\newcommand{\kk}{\kappa}

\newcommand{\cO}{\mathcal{O}}
\newcommand{\Gal}{\mathop{\mathrm{Gal}}\nolimits}

\newcommand{\Res}{\mathop{\mathrm{Res}}\nolimits}
\newcommand{\rad}{\mathop{\mathrm{rad}}\nolimits}
\newcommand{\soc}{\mathop{\mathrm{soc}}\nolimits}

\newcommand{\Ppow}{\mathop{\uparrow{}p}}

\newcommand{\sfr}[2]{\nicefrac{#1}{#2}}
\newcommand{\I}[1]{\llbracket #1\rrbracket}

\newcommand{\Din}{\rotatebox{270}{$\in$}}

\newcounter{enumi_saved}

\newcommand{\saveenum}{\setcounter{enumi_saved}{\value{enumi}}}
\newcommand{\loadenum}{\setcounter{enumi}{\value{enumi_saved}}}

\makeatletter
\def\imod#1{\allowbreak\mkern10mu({\operator@font mod}\,\,#1)}
\makeatother

\allowdisplaybreaks

\begin{document}

\title[Polynomials generating cyclic extensions over $\kp$-adic fields]{A characterization of
  Eisenstein polynomials generating cyclic extensions of degree $p^2$ and $p^3$ over an unramified
  $\kp$-adic field} \author{Maurizio Monge}
\email{maurizio.monge@sns.it}
\address{Scuola Normale Superiore di Pisa - Piazza dei Cavalieri, 7 -
  56126 Pisa}
\date{\today}


\begin{abstract}
  Let $p\neq2$ be a prime. We show a technique based on local class field theory and on the
  expansions of certain resultants which allows to recover very easily Lbekkouri's characterization
  of Eisenstein polynomials generating cyclic wild extensions of degree $p^2$ over $\Q_p$, and to
  extend it to the case of the base field $K$ being an unramified extension of $\Q_p$.

  Furthermore, when a polynomial satisfies only some of the stated conditions, we show that the
  first unsatisfied condition gives information about the Galois group of the normal closure. This
  permits to give a complete classification of Eisenstein polynomials of degree $p^2$ whose
  splitting field is a $p$-extension, providing a full description of the Galois group and its
  higher ramification subgroups.

  We then apply the same methods to give a characterization of Eisenstein polynomials of degree
  $p^3$ generating a cyclic extension.

  In the last section we deduce a combinatorial interpretation of the monomial symmetric function
  evaluated in the roots of the unity which appear in certain expansions.
\end{abstract}

\maketitle

\section{Introduction}

In this paper we explore the techniques which can be used to deduce necessary and sufficient
conditions for a polynomial to have a certain Galois group as group of the splitting field over a
$\kp$-adic field.

Lbekkouri gave in \cite{Lbekkouri2009a} congruence conditions for Eisenstein polynomials of degree
$p^2$ with coefficients in the rational $\kp$-adic field $\Q_p$, which are satisfied if and only if
the generated extension is Galois. Since the multiplicative group $U_{1,\Q_p}$ of $1$-units of
$\Q_p$ has rank $1$ as $\Z_p$-module and in particular
$\sfr{U_{1,\Q_p}(\Q_p^\times)^p}{(\Q_p^\times)^p}\cong\sfr{\Z}{p\Z}$, we have by local class field
theory that every Galois totally ramified extension of degree $p^2$ over $\Q_p$ is cyclic, and
consequently over $\Q_p$ the problem is reduced to finding conditions for Eisenstein polynomials of
degree $p^2$ to generate a cyclic extension.

If the base field $K$ is a proper extension of $\Q_p$ this is no longer true, so the restriction of
considering polynomials which generate cyclic extensions has to be added explicitly. If $K$ is
ramified over $\Q_p$ even the characterization of the possible upper ramification jumps is a
non-trivial problem (see \cite{Maus1971,Miki1981}) and the problem seems to be very difficult for a
number of other reasons, so we will only consider fields $K$ which are finite unramified extensions
over $\Q_p$, with residue degree $f=f(K/\Q_p)=[K:\Q_p]$. In this setting the problem is still
tractable without being a trivial generalization of the case over $\Q_p$, and we will show a
technique which allows to handle very easily the case of degree $p^2$.

During the proof the Artin-Hasse exponential function comes into play, and we use it to clarify the
connection between the image of the norm map and the coefficients of the Eisenstein polynomial.

While some condition are necessary to force the splitting field to be a $p$-extension, the remaining
conditions can be tested in order, and the first which fails gives information on the Galois group
of the splitting field. Taking into account another family of polynomial which can never provide a
cyclic extension of degree $p^2$, we give a full classification of the polynomials of degree $p^2$
whose normal closure is a $p$-extension, providing a complete description of the Galois group of the
normal closure with its ramification filtration, see \cite{caputo2007classification} for an abstract
classification of all such extensions when the base field is $\Q_p$.

We then show how the same methods apply to characterize Eisenstein polynomials of degree $p^3$
generating a cyclic extension. This case a substantially more complicated is obtained, but the
strategy used in degree $p^2$ can still be applied in a relatively straightforward way. We plan to
show in a forthcoming paper the how the techniques used here can be extended to give a
characterization of polynomials generating all remaining groups of order $p^3$, including the
non-abelian ones.

In the last section we give a combinatorial interpretation of certain sums of roots of the unity
which appear during the proof, it is actually much more than needed but it has some interest on its
own.

\subsection{Acknowledgements}
We feel indebted with Luca Caputo and Fran\c{c}ois Laubie for some fruitful discussions and for the
motivation to try to recover and generalize Lbekkouri's result. We would also like to thank Philippe
Cassou-Nogu\`es, Ilaria Del Corso, Roberto Dvornicich and Boas Erez for various discussions on this
topic, and the Institut de Math\'ematiques de Bordeaux for hospitality while conceiving this work.

\section{Preliminaries}
Let $K$ be a $\kp$-adic field, we denote by $\cO_K$ the ring of integers, by $\kp_K$ the maximal
ideal, by $U_{i,K}$ the $i$-th units group, by $\kk_K$ the residue field $\sfr{\cO_K}{\kp_K}$, and
for $x\in{}\cO_K$ let $\bar{x}$ be its image in $\kk_K$. We will denote with $[K^\times]_K$ the
group of $p$-th power classes $\sfr{K^\times}{(K^\times)^p}$. For integers $a,b$, it will be
convenient to denote by $\I{a,b}$ the set of integers $a\leq{}i\leq{}b$ such that $(i,p)=1$.

We start computing modulo which power of $p$ we must consider the coefficients of an Eisenstein
polynomial (this computation is very well known and we repeat it only to keep the paper
self-contained, see \cite{Krasner1962} for more details): let $f(X)=\sum_{i=0}^nf_{n-i}X^i$ and
$g(X)=\sum_{i=0}^n g_{n-i}X^i$ be Eisenstein polynomials of degree $n$ say, $\rho$ a root of $g$,
$\pi=\pi_1,\pi_2,\dots$ the roots of $f$ with $\pi$ the most near to $\rho$, and put $L=K(\pi)$. Let
$v$ be the biggest lower ramification jump and $\sD_f=f'(\pi)$ be the different, if
\[
   \left|(f_{n-i}-g_{n-i})\pi^i\right| < \left|\pi^{v+1}\sD_f\right|
\]
then being
\[
   f(\rho) = f(\rho)-g(\rho) = \sum_{i=0}^n \left(f_{n-i}-g_{n-i}\right)\rho^i
\]
we obtain $|f(\rho)|<|\pi^v\sD_f|$. We have
\[
  \left|(\rho-\pi)\cdot \prod_{i=2}^n(\pi-\pi_i)\right| \leq
  \left|\prod_{i=1}^n(\rho-\pi_i)\right| < \left|\pi^{v+1}\sD_f\right|,
\]
in fact $|\pi-\pi_i|\leq|\rho-\pi_i|$ for $i\geq2$ or we would contradict the choice of $\pi=\pi_1$.
Consequently $|\rho-\pi|<|\pi^{v+1}|$ which is equal to the minimum of the $|\pi-\pi_i|$, and hence
$K(\rho)\subseteq{}K(\pi)$ by Krasner's lemma, and $K(\rho)=K(\pi)$ having the same degree.

Let now $K$ be unramified over $\Q_p$, then $U_{1,K}^{p^i}=U_{i+1,K}$, and consequently by local
class field theory the upper ramification jumps of a cyclic $p$-extension are $1,2,3,\dots$, and the
lower ramification jumps are $1,p+1,p^2+p+1,\dots$.

For an extension of degree $p^k$ with lower ramification jumps $t_0\leq{}t_1\leq\dots\leq{}t_{k-1}$
we can compute $v_L(\sD_{L/K})$ as $\sum_{i=1}^k (p^i-p^{i-1})t_{k-i}$, which for a cyclic $L/K$ of
degree $p^2$ or $p^3$ is $3p^2-p-2$ (resp. $4p^3-p^2-p-2$), while $v_L(\pi^{v+1}\sD_{L/K})$ is
respectively $3p^2=v_L(p^3)$ and $4p^3=v_L(p^4)$. Hence we obtain the condition on the precision of
the coefficients, which we state in a proposition for convenience:

\begin{prop} \label{prop_coeff}
  Let $L/K$ be a totally ramified cyclic extension of degree $n=p^2$ (resp. $n=p^3$) determined by
  the Eisenstein polynomial $f(X)=\sum_{i=0}^nf_{n-i}X^n$. Then the lower ramification jumps are
  $1,p+1$ (resp. $1,p+1,p^2+p+1$), $v_L(\sD_{L/K})$ is equal to $3p^2-p-2$ (resp. is
  $4p^3-p^2-p-2$), and the extension is uniquely determined by the classes of $f_n\imod{p^4}$ and
  $f_i\imod{p^3}$ for $0\leq{}i<n$ (resp. by the classes of $f_n\imod{p^5}$ and $f_i\imod{p^4}$ for
  $0\leq{}i<n$, for $n=p^3$).
\end{prop}

\subsection{Additive polynomials}
We will need a few facts about additive polynomials, and in particular some formulas to express in
terms of the coefficients the condition that and additive polynomial has range contained in the
range of some other additive polynomial. We resume what we need in the following

\begin{prop}
\label{prop_add}
Let $A(Y)=a_pY^p+a_1Y$ be an additive polynomial in $\kk_K[Y]$ such
  that $A'(0)\neq0$ and all the roots of $A(Y)$ are in $\kk_K$, and let $B(Y)=b_pY^p+b_1Y$,
  $C(Y)=c_{p^2}Y^{p^2}+c_pY^p+c_1Y$ and $D(Y)=d_{p^3}Y^{p^3}+d_{p^2}Y^{p^2}+d_pY^p+d_1Y$ be any
  three other additive polynomials in $\kk_K[Y]$. Then
\begin{itemize}[leftmargin=1.5em]
\item $B(\kk_K)\subseteq A(\kk_K)$ if and only if $b_p=a_p(\sfr{b_1}{a_1})^p$, and in this case
  $B(Y)$ is equal to $A(\sfr{b_1}{a_1}Y)$,
\item $C(\kk_K)\subseteq A(\kk_K)$ if and only if
  $c_p=a_p(\sfr{c_1}{a_1})^p+a_1(\sfr{c_{p^2}}{a_p})^{\sfr{1}{p}}$, and in this case $C(Y)$ can be
  written as $A(\beta{}Y^p+\sfr{c_1}{a_1}Y)$ with $\beta=(\sfr{c_{p^2}}{a_p})^{\sfr{1}{p}}$ or
  equivalently $\beta=\sfr{c_p}{a_1}-\sfr{a_p}{a_1}(\sfr{c_1}{a_1})^p$.
\item $D(\kk_K)\subseteq A(\kk_K)$ if and only if
  $\sfr{a_1}{a_p}(\sfr{d_{p^3}}{a_p})^{\sfr{1}{p}}+(\sfr{d_{p}}{a_1})^p=\sfr{d_{p^2}}{a_p}+
  (\sfr{a_p}{a_1})^p(\sfr{d_1}{a_1})^{p^2}$.
\end{itemize}
\end{prop}

Note that being $\kk_K$ finite and hence perfect the map $x\mapsto{}x^p$ is an automorphism, and we
just denote by $x\mapsto{}x^{\sfr{1}{p}}$ the inverse automorphism.

\begin{proof}
  Since $A'(0)\neq0$ and all the roots of $A(Y)$ are in $\kk_K$ we have from the theory of additive
  polynomials (see \cite[Chap. 5, \S2, Corollary 2.4]{fesenko2002local}) that if
  $B(\kk_K)\subseteq{}A(\kk_K)$ then $B(Y)=A(G(Y))$ for some other additive polynomial $G(Y)$ which
  should be linear considering the degrees, $G(Y)=\alpha{}Y$ say. Consequently it has to be
  $B(Y)=a_p\alpha^p{}Y^p+a_1\alpha{}Y$, and comparing the coefficients we obtain that
  $\alpha^p=(\sfr{b_1}{a_1})^p$ and should also be equal to $\sfr{b_p}{a_p}$. Similarly if
  $C(\kk_K)\subseteq{}A(\kk_K)$ we should have
\[
  C(Y)=A(\beta{}Y^p+\alpha{}Y)=a_p\beta^p{}Y^{p^2}+(a_p\alpha^p+a_1\beta)Y^p+a_1\alpha{}Y,
\]
and we deduce $\alpha=\sfr{c_1}{a_1}$, $\beta^p=\sfr{c_{p^2}}{a_p}$, and we obtain the condition
substituting $\alpha,\beta$ in $c_p=a_p\alpha^p+a_1\beta$. If $D(\kk_K)\subseteq{}A(\kk_K)$ then
$D(Y)$ should be $A(\gamma{}Y^{p^2}+\beta{}Y^p+\alpha{}Y)$ and hence
\[
a_p\gamma^p{}Y^{p^3}+(a_p\beta^p+a_1\gamma){}Y^{p^2}+(a_p\alpha^p+a_1\beta)Y^p+a_1\alpha{}Y,
\]
$\alpha=\sfr{d_1}{a_1}$, $\gamma=(\sfr{d_{p^3}}{a_p})^{\sfr{1}{p}}$, and $\beta^p$ can be written in
two different ways as
\[
     \sfr{d_{p^2}}{a_p}-\sfr{a_1}{a_p}(\sfr{d_{p^3}}{a_p})^{\sfr{1}{p}} = 
      (\sfr{d_{p}}{a_1} - \sfr{a_p}{a_1}(\sfr{d_1}{a_1})^p)^p.
\]
The condition is clearly also sufficient.
\end{proof}

The following proposition will also be useful, it gives a criterion to verify if the splitting field
of an additive polynomial of degree $p^2$ is a $p$-extension (that is, either trivial of cyclic of
degree $p$) which is slightly easier to test than the condition itself:

\begin{prop}
\label{add_gal}
  Let $A(Y)=Y^{p^2}+aY^p+bY$ be an additive polynomial in $\kk_K[Y]$, than the splitting field is a
  $p$-extension over $\kk_K$ precisely when $A(Y)$ has a root in $\kk_K^\times$, and
  $b\in(\kk_K^\times)^{p-1}$.
\end{prop}
\begin{proof}
  If the Galois group is a $p$-group then its orbits have cardinality divisible by $p$ and the
  action on the roots of $A(Y)$ should clearly have some fixed point other than $0$,
  $\eta\in\kk_K^\times$ say. If $\beta=\eta^{p-1}$ then the roots of $Y^p-\beta{}Y$ are roots of
  $A(Y)$, so by \cite[Chap. 5, \S2, Prop. 2.5]{fesenko2002local} $A(Y)$ is $B(Y^p-\beta{}Y)$ for
  some additive polynomial $B(Y)$ which has to be monic too, $B(Y)=Y^p-\alpha{}Y$ say. The roots of
  $B(Y)$ have to be in $\kk_K$ or it, and hence $A(Y)$, would generate an extension of order prime
  with $p$, and consequently $\alpha$ has to be in $(\kk_K^\times)^{p-1}$, and $b=\alpha\beta$ as
  well. On the other hand if a root $\eta$ is in $\kk_K$ we can write $A(Y)=B(Y^p-\beta{}Y)$ for
  $\beta=\eta^{p-1}$, and replacing $Y$ by $\eta{}Z$ we can consider $B(\eta^p(Z^p-Z))$, and the
  extension is obtained as an Artin-Schreier extension over the extension determined by
  $B(Y)$. Consequently we only need the extension determined by $B(Y)$ to be trivial, which is the
  case when $b\in(\kk_K^\times)^{p-1}$.
\end{proof}

\subsection{Sum of roots of the unity}
Let $\zeta_\ell$ be a primitive $\ell$-th root of the unity for some $\ell$, we define for each
tuple $\lambda=(\lambda_1,\lambda_2,\dots,\lambda_r)$ of $r$ integers the sum
\[
    \Sigma_\lambda(\ell) = \sum_{\iota=(\iota_1,\dots,\iota_r)} 
       \zeta_\ell^{\iota_1\lambda_1+\iota_2\lambda_2+\dots+\iota_r\lambda_r},
\]
where the sum ranges over all the $r$-tuples $\iota=(\iota_1,\dots,\iota_r)$ such that
$0\leq\iota_i\leq\ell-1$ for each $i$, and the $\iota_i$ are all distinct.

We deduce some property of the sums $\Sigma_\lambda(\ell)$ to help expanding the expressions that
will appear. For each $\lambda=(\lambda_1,\lambda_2,\dots)$ and integer $k$ put $k\lambda$ for the
partition $(k\lambda_1,k\lambda_2,\dots)$. For integers $\ell,k,m$ let's define the functions
\[
   \delta^{[m]}_{\ell,k} =
     \begin{cases}
        \ell & \text{ if }\ell\geq{}m\text{ and }\ell\mid{}k, \\
        0    & \text{in any other case,}
     \end{cases}
\]
and put $\delta_{\ell,k}=\delta^{[1]}_{\ell,k}$ for short. Then we have

\begin{prop}
  \label{prop_sums}
  Assume $(\ell,p)=1$, $\ell>1$. For any partition $\lambda$ we have
  $\Sigma_{p\lambda}(\ell)=\Sigma_{\lambda}(\ell)$. For $k\geq1$ we have
  $\Sigma_{(k)}(\ell)=\delta_{\ell,k}$, and $\Sigma_{(k,1)}(\ell)=\delta^{[2]}_{\ell,k+1}$, and if
  $(k,p)=1$ we also have $\Sigma_{(k,p)}(\ell)=\delta^{[2]}_{\ell,k+p}$ and
  $\Sigma_{(k,p^2)}(\ell)=\delta^{[2]}{\ell,k+p^2}$. We also have
  $\Sigma_{(1,1,1)}(\ell)=\delta^{[3]}_{\ell,3}$, $\Sigma_{(p,1,1)}(\ell)=\delta^{[3]}_{\ell,p+2}$ and
  $\Sigma_{(p,p,1)}(\ell)=\delta^{[3]}_{\ell,2p+1}$.
\end{prop}

The proof can be obtained via an easy computation, but we omit it being also an immediate
consequence of Lemma \ref{combinat} proved in the last section.

\section{Polynomials of degree $p^2$ generating a cyclic extension}

Since the different $f'(\pi)$ has valuation $3p^2-p-2$ it must come from a term $f_{p+1}X^{p^2-p-1}$
with $v_p(f_{p+1})=2$, and we must have $v_p(f_i)\geq2$ for all $(i,p)=1$, and $v_p(f_i)\geq3$ if
furthermore $i>p+1$.

Since the first jump is at $1$, the $p$-th coefficient of the ramification polynomial $f(X+\pi)$
needs to have valuation exactly equal to $(p^2-p)\cdot2=2p^2-2p$, and observe that a monomial
$f_{p^2-i}(X+\pi)^i$ contributes at most one term $\binom{i}{p}f_{p^2-i}\pi^{i-p}X^p$ in $X^p$. The
valuations of these terms have different remainders modulo the degree $p^2$, and consequently the
minimal valuation of the $\binom{i}{p}f_{p^2-i}\pi^{i-p}$ has to be $2p^2-2p$, is achieved for
$i=p^2-p$ and we must have $v_p(f_p)=1$, while $v_p(f_{pk})\geq2$ for all $2\leq{}k\leq{}p-1$.

So we have to respect the following

\begin{condition}
We must have
\begin{itemize}
\item $v_p(f_{p})=1$, and $v_p(f_{pi})\geq2$ for $i\in\I{2,p-1}$,
\item $v_p(f_{i})\geq2$ for $i\in\I{1,p-1}$, $v(f_{p+1})=2$ and $v_p(f_{i})\geq3$ for
  $i\in\I{p+2,p^2-1}$.
\end{itemize}
\end{condition}

In other words, turning to $0$ all the $f_i$ divisible by $p^3$ for $i>1$, that we are allowed to
do by Prop. \ref{prop_coeff}, $f(X)$ can be written as
\begin{equation} \label{f1}
  f(X) = X^{p^2} + \underbrace{f_{p}X^{p^2-p} +\phantom{\bigg|}\! f_{p^2}}_{\substack{\Din\\p\cO[X]}}
    + \underbrace{\sum_{j\in\I{2,p-1}} f_{pj}X^{p^2-pj}
    + \sum_{k\in\I{1,p+1}} f_{k} X^{p^2-k}}_{\substack{\Din\\p^2\cO[X]}}.
\end{equation}

Suppose now $L$ to be an arbitrary extension determined by a root $\pi$ of the polynomial $f(X)$, by
local class field theory it is a totally ramified abelian extension precisely when
$N_{L/K}(L^\times)\cap{}U_{1,K}=N_{L/K}(U_{1,L})$ has index $p^2$ in $U_{1,K}$ and the corresponding
quotient is cyclic.

Being $U_{i+1,K}=U_{1,K}^{p^{i}}$ for $i\geq1$, to have a cyclic extension $N_{L/K}(U_{1,L})U_{2,K}$
shall have index $p$ in $U_{1,K}$, and $N_{L/K}(U_{1,L})\cap{}U_{2,K}$ index $p$ in $U_{2,K}$.

Let's recall that for each $i\geq0$ we have a natural map
$\times{}p:\sfr{\kp^i_K}{\kp^{i+1}_K}\rightarrow\sfr{\kp^{i+1}_K}{\kp^{i+2}_K}$ induced by
multiplication by $p$, and for $i\geq1$ being
$(1+\theta{}p^i)^p\equiv1+\theta{}p^{i+1}+\cO(p^{i+2})$ we have a natural map
$\Ppow:\sfr{U_{i,K}}{U_{i+1,K}}\rightarrow\sfr{U_{i+1,K}}{U_{i+2,K}}$ induced by
  taking $p$-th powers and a commutative diagram
\begin{equation} \label{p_maps}
\begin{array}{c}
\xymatrix@!0@=25pt{
\sfr{\kp^i_K}{\kp^{i+1}_K} \ar@{->}[rrr]^{\times p}
       \ar@{->}[dd]_{\mu_i}  & & & \sfr{\kp^{i+1}_K}{\kp^{i+2}_K}\ar@{->}[dd]_{\mu_{i+1}} \\ \\  
 \sfr{U_{i,K}}{U_{i+1,K}} \ar@{->}[rrr]^{\Ppow}  & & & \sfr{U_{i+1,K}}{U_{i+2,K}}
}\end{array}
\end{equation}
where $\mu_i$ is induced by $x\mapsto1+x$.

Since $N_{L/K}(U_{1,L})\cap{}U_{2,K}$ will certainly contain $N_{L/K}(U_{1,L})^pU_{3,K}$ which has
index $p$ in $U_{2,K}$ and consequently has to be equal, for $L/K$ to be Galois cyclic we need
\begin{equation} \label{normc}
   N_{L/K}(U_{1,L})\subseteq{}1+pV,\qquad N_{L/K}(U_{1,L})\cap{}U_{2,K}\subseteq{}1+p^2V
 \end{equation}
 for some $\F_p$-subspace $V$ of $\sfr{\cO_K}{\kp_K}$ of codimension $1$. Note that $V$ is uniquely
 determined by $N_{L/K}(U_{1,L})U_{2,K}$ as a subgroup of $U_{1,L}$, we commit the abuse of denoting
 also as $V$ its preimage in $\cO_L$, and the meaning of the expressions $1+p^iV$ should be clear.

 If $i\geq1$ then $N_{L/K}(U_{i+1,L})\subseteq{}U_{\phi_{L/K}(i)+1,K}$ (see \cite[Chap. 3, \S3.3\
 and \S3.4]{fesenko2002local}), and in our case we have $N_{L/K}(U_{2,L})\subseteq{}U_{2,K}$ and
 $N_{L/K}(U_{p+2,L})\subseteq{}U_{3,K}$.

 Consequently we can prove that $L/K$ is Galois by showing that the norms of elements whose images
 generate $\sfr{U_{1,L}}{U_{2,L}}$ are contained in $1+pV$ for some $V$, and that for any $x$
 obtained as combination of a set of elements whose images generate $\sfr{U_{1,L}}{U_{p+2,L}}$ and
 such that $N_{L/K}(x)\in{}U_{2,K}$ we actually have $N_{L/K}(x)\in{}1+p^2V$. We will take as
 generators the elements of the form $(1-\theta\pi^\ell)$ for $\ell\in\I{1,p+1}$, plus those of the
 forms $(1-\theta\pi)^p$ for $\theta$ in the set of multiplicative representative. Those of the form
 $(1-\theta\pi)^p$ can be discarded considering that we are already requesting
 $N_{L/K}(1-\theta\pi)\in1+pV$, so their norm will certainly be in $1+p^2V$.

The norm of an element of the form $1-\theta\pi^\ell$ can be expressed as
\[
   N_{L/K}(1-\theta\pi^\ell) = \prod_{\pi_i \mid f(\pi_i)=0}(1-\theta\pi_i^\ell) 
  = \Res_X(1-\theta{}X^\ell, f(X)),
\]
where $\pi_i$ are the roots of $f(X)$ and we denote by $\Res_X$ the resultant in $X$.

For a polynomial $a(X)$ of degree $d$ let's denote by $\tilde{a}(X)$ the conjugate polynomial
$X^da(X^{-1})$, then for each pair of polynomials $a(X), b(X)$ we have
$Res_X(a(X),b(X))=Res_X(\tilde{b}(X),\tilde{a}(X))$. 

Consequently $N_{L/K}(1-\theta\pi^\ell)$ can also we written as
\[
\Res_X(\tilde{f}(X),X^\ell-\theta) 
  = \prod_{i=0}^{\ell-1} \tilde{f}(\zeta_{\ell}^i\theta^{\sfr{1}{\ell}})
\]
for some primitive $\ell$-th root of the unity. In the expansion of the second term only integral
powers of $\theta$ appear being invariant under the substitution
$\theta^{\sfr{1}{p}}\rightarrow\zeta_{\ell}\theta^{\sfr{1}{\ell}}$. In the same way while the terms
in the right hand side live in $K(\zeta_\ell)$ the result always lives in $K$, and the above
expansion should be rather considered as a combinatorial expedient.

Put $T=\theta^{\sfr{1}{\ell}}$ and consider it as an indeterminate, from the expression for
$f(X)$ in the \eqref{f1} we have that $N_{L/K}(1-\theta\pi^\ell)$ is
\begin{equation*}
  \prod_{i=0}^{\ell-1} \bigg(1 + \underbrace{f_{p}\zeta_\ell^{ip}T^p 
        +\phantom{\bigg|}\! f_{p^2}\zeta_\ell^{ip^2}T^{p^2} }_{\substack{\Din\\\kp_K}}
    + \underbrace{\sum_{j\in\I{2,p-1}} f_{pj}\zeta_\ell^{ipj}T^{pj}
    + \sum_{k\in\I{1,p+1}} f_{k} \zeta_\ell^{ik}T^{k}}_{\substack{\Din\\\kp_K^2}} \bigg).
\end{equation*}

For each tuple $\lambda=(\lambda_1,\lambda_2,\dots,\lambda_r)$ of $r$ integers put
$f_{\lambda}T^{|\lambda|}$ for the term $\prod_{i=1}^rf_{\lambda_i}T^{\lambda_i}$ of the expansion,
its coefficient is computed over all the ways we can partitioning the
$1,\zeta_\ell,\zeta_\ell^2,\dots$ in sets of cardinality $m_k(\lambda)$, the number of parts
$\lambda_i$ equal to $k$, for all $k$. Computing the ratio to the ordered choices of $r$ distinct
elements, which is the collection over which we are iterating while computing
$\Sigma_\lambda(\ell)$, we have that the coefficient of $f_\lambda{}T^{|\lambda|}$ in the expansion
is $\frac{1}{\prod_{k\geq1}m_k(\lambda)!}\cdot\Sigma_\lambda(\ell)$.

In particular we have that discarding the terms with valuation $\geq3$ and subtracting $1$ the above
product can be expanded modulo $p^3$ as

\begin{align*}
  \kp_K\ni &\left[\qquad
\begin{gathered}
  \Sigma_{(p)}(\ell)\cdot{}f_{p}T^p + \Sigma_{(p^2)}(\ell)\cdot{}f_{p^2}T^{p^2}
  \end{gathered}
\right. \\
  \kp_K^2\ni &\left[
  \begin{gathered}
  +\frac{1}{2}\Sigma_{(p,p)}(\ell)\cdot{}f_{p}^2T^{2p}
  +\Sigma_{(p^2,p)}(\ell)\cdot{}f_{p}f_{p^2}T^{p^2+p}
  +\frac{1}{2}\Sigma_{(p^2,p^2)}(\ell)\cdot{}f_{p^2}^2T^{2p^2} \\
  +\sum_{j\in\I{2,p-1}} \left( \Sigma_{(pj)}(\ell)\cdot f_{pj}T^{pj} \right)
  +  \sum_{k\in\I{1,p+1}} \left( \Sigma_{(k)}(\ell)\cdot f_{k} T^{k}\right),
  \end{gathered}\right.
\end{align*}
which applying Prop. \ref{prop_sums} can be rewritten as

\begin{equation}
 \label{expa1}
\begin{aligned}
  \kp_K\ni &\left[\qquad
\begin{gathered}
  \qquad\delta_{\ell,1}f_{p}T^p + \delta_{\ell,1}f_{p^2}T^{p^2}
  \end{gathered}
\right. \\
  \kp_K^2\ni &\left[
  \begin{gathered}
  -\frac{1}{2}\delta^{[2]}_{\ell,2}f_{p}^2T^{2p}
  -\delta^{[2]}_{\ell,p+1}f_{p}f_{p^2}T^{p^2+p}
  -\frac{1}{2}\delta^{[2]}_{\ell,2}f_{p^2}^2T^{2p^2} \\
  +\sum_{j\in\I{2,p-1}} \delta_{\ell,j}f_{pj}T^{pj}
  +  \sum_{k\in\I{1,p+1}} \delta_{\ell,k} f_{k} T^{k}.
  \end{gathered}\right.
\end{aligned}
\end{equation}

The expansion for $\ell=1$ and modulo $p^2$ tells us that the norms in $U_{1,K}$ are of the form
$1+f_pT^{p}+f_{p^2}T^{p^2}+\cO(p^2)$ for some $T$, consequently put $F_p=\overline{\sfr{f_{p}}{p}}$,
$F_{p^2}=\overline{\sfr{f_{p^2}}{p}}$ and consider the additive polynomial
\begin{equation} \label{cond1}
   A(Y)= F_{p^2}Y^p+F_pY
\end{equation}
over the residue field. It defines a linear function, if $V$ is the range $A(\kk_K)$ then
$N_{L/K}(U_{1,L})U_{2,K}$ is contained in $1+pV$, and $V$ has codimension $1$ precisely when the map
defined by $A$ has a kernel of dimension $1$, that is when $-\sfr{F_p}{F_{p^2}}$ is a $(p-1)$-th
power (if $K=\Q_p$ we shall have $V=0$ and the condition is $F_{p^2}=-F_p$).

\begin{condition}
  We must have $-\sfr{F_p}{F_{p^2}}\in{}\kk_K^{p-1}$.
\end{condition}

Now for $\ell\geq2$ the first part of the expansion \eqref{expa1} is $0$, while the remaining part
shall be contained in $p^2V$ for each $\ell$ and each specialization of $T^\ell=\theta$.

Note that we only consider the $\ell$ prime with $p$, and all the $\delta^{[m]}_{\ell,i}$ appearing
for some $m,i$ have been reduced with no loss of generality to have $(i,p)=1$.  Consequently for
$\ell\geq2$ the expansion can be written as a polynomial
$C_\ell(T^\ell)=\ell\sum_{(k,p)=1}c_{k\ell}(T^{k\ell})$, where for each $\ell$ prime with $p$ we
denote by $c_\ell(T^\ell)$ the polynomial of $T^\ell$ obtained by the evaluating equation
\eqref{expa1}, but changing the definition of $\delta^{[m]}_{a,b}$ to be $1$ if $a=b$, and $0$ if
$a\neq{}b$.

Fix $\ell$, then $c_\ell(T^\ell)$ can be obtained via a sort of M\"obius inversion
\begin{align*}
   \sum_{(k,p)=1} \mu(k) \frac{C_{k\ell}(T^{k\ell})}{k\ell}
    &= \sum_{(k,p)=1}\left( \mu(k) \cdot \sum_{(j,p)=1} c_{jk\ell}(T^{jk\ell}) \right) \\
    &= \sum_{(i,p)=1} \left(c_{i\ell}(T^{i\ell})\cdot \sum_{k\mid{}i}\mu(k) \right)
\end{align*}
by change of variable $i=jk$, obtaining $c_\ell(T^\ell)$ by the properties of the M\"obius function
$\mu$. In view of the isomorphism $\sfr{\kp^2_K}{\kp^3_K}\rightarrow\sfr{U_{2,K}}{U_{3,K}}$ induced
by $x\mapsto{}1+x$ and specializing the argument $T^{k\ell}$ of $C_{k\ell}(T^{k\ell})$ to
$\theta'=\frac{1}{\ell}\theta^k$ we have that
\[
  1+\frac{1}{k\ell}C_{k\ell}(\theta^k)\equiv 
   N_{L/K}\left(1-\frac{1}{\ell}\theta^k\pi^{k\ell}\right)^{\sfr{1}{k}} \imod{p^3},
\]
for each $\ell,k$ prime with $p$ and each $\theta=T^\ell$. Consequently $1+c_\ell(\theta)$ is
congruent modulo $p^3$ to the norm of
\[
   \prod_{(k,p)=1}\left(1-\frac{1}{\ell}\theta^k\pi^{k\ell}\right)^{\sfr{\mu(k)}{k}} = E(\theta\pi^\ell),
\]
where $E(x)$ is the Artin-Hasse exponential function (in its original form, according to
\cite[Chap. 3, \S9.1]{fesenko2002local}). Note that we can equivalently require that all the
$N_{L/K}(E(\theta\pi^\ell))$ are in $1+p^2V$, for $\ell\in\I{2,p+1}$ and residue representative
$\theta\in{}K$.

Put $A_\ell(Y)=\overline{\sfr{c_\ell(Y)}{p^2}}$, we obtain depending on $\ell$ the additive
polynomials
\begin{gather*}
   \qquad\qquad -F_pF_{p^2}Y^p + G_{p+1}Y \qquad\qquad\qquad\qquad \ell=p+1, \\
  \qquad\qquad\qquad G_{p\ell}Y^p + G_\ell Y   \qquad \qquad\qquad\qquad \ell\in\I{3, p-1},\\
  -\frac{1}{2}F_{p^2}^2Y^{p^2}+\left(G_{2p}-\frac{1}{2}F_p^2\right)Y^p+G_{2}Y \qquad \ell=2,
\end{gather*}
where for convenience we have put $G_i=\overline{\sfr{f_i}{p^2}}$ for each $i\neq{}p,p^2$. 

Hence we have obtained the

\begin{condition}
  For each $\ell\in\I{2,p+1}$ we shall have $A_\ell(\kk_K)\subseteq{}A(\kk_K)$.
\end{condition}

We are left to ensure that norms are in $1+p^2V$ for $\ell=1$ with $T=\theta$ for some $\theta$ such
that $\theta^{p^2-p}\equiv-\sfr{f_p}{f_{p^2}}\imod{p}$. Consider again the \eqref{expa1},
considering the definition of the $c_k(T)$ we have that $\sum_{(k,p)=1}c_{k}(T^{k})$ differs from
$C_1(T)=N_{L/K}(1-T\pi)-1$ by the extra term
\[
- \frac{1}{2}f_{p}^2T^{2p} - f_{p}f_{p^2}T^{p^2+p}
   - \frac{1}{2}f_{p^2}^2T^{2p^2} = -\frac{1}{2}\left(f_{p}T^{p} + f_{p^2}T^{p^2}\right)^2,
\]
which is however even contained in $\kp^4$ for $T=\theta$. Since we already required the polynomials
$c_k(Y^{k})$ to take values in $p^2V$ identically for $k\geq2$, our requirement becomes that
\begin{equation*}
  c_1(\theta) = f_{p^2}\theta^{p^2} + f_{p}\theta^{p} + f_{1}\theta
\end{equation*}
shall be contained in $p^2V$ too. Hence we have the

\begin{condition}
  Let $\theta$ be such that $\theta^{p^2-p}\equiv-\sfr{f_p}{f_{p^2}}\imod{p}$, then we must have
  $\overline{\sfr{c_1(\theta)}{p^2}}\in{}V$.
\end{condition}

Collecting all the above conditions and applying Prop. \ref{prop_add} to obtain conditions on the
coefficients we have the following theorem:

\begin{theo}
\label{theo1}
The Eisenstein polynomial $f(X)=X^{p^2}+f_1X^{p^2-1}+\dots+f_{p^2-1}X+f_{p^2}$ determines a Galois
extension of degree $p^2$ over $K$ if and only if
\begin{enumerate}[leftmargin=2.5em]
\item\label{tcr1} $v_p(f_{p})=1$, and $v_p(f_{pi})\geq2$ for $i\in\I{2,p-1}$,
\item\label{tcr2} $v_p(f_{i})\geq2$ for $i\in\I{1,p-1}$, $v(f_{p+1})=2$ and $v_p(f_{i})\geq3$ for
  $i\in\I{p+2,p^2-1}$,
\saveenum
\end{enumerate}
putting $F_p=\overline{\sfr{f_p}{p}}$, $F_{p^2}=\overline{\sfr{f_{p^2}}{p}}$, and
$G_i=\overline{\sfr{f_i}{p^2}}$ for all $i\neq{}p,p^2$ we have
\begin{enumerate}[leftmargin=2.5em]
\loadenum
\item\label{tcgal} $-\sfr{F_p}{F_{p^2}}\in{}\kk_K^{p-1}$,
\item\label{tcpp1} $G_{p+1}^p=-F_p^{p+1}$,
\item\label{tcell} $G_{p\ell}=F_{p^2}\left(\sfr{G_\ell}{F_p}\right)^p$, for all $\ell\in\I{3,p-1}$,
\item\label{tc2} $G_{2p}=F_{p^2}\left(\sfr{G_2}{F_p}\right)^p+\frac{1}{2}F_p\left(F_p-F_{p^2}^{\sfr{1}{p}}\right)$,
\saveenum
\end{enumerate}
if $\theta$ is such that $\bar\theta^{p(p-1)}=-\sfr{F_p}{F_{p^2}}$ we have that (for any choice of $\theta$)
\begin{enumerate}[leftmargin=2.5em]
\loadenum
\item\label{tc1}
  $F_{p^2}X^p+F_pX-\overline{\frac{1}{p^2}\left(f_{p^2}\theta^{p^2}+f_{p}\theta^{p}\right)}-G_1\bar\theta$
  has a root in $\kk_K$.
\end{enumerate}
\end{theo}

\section{Polynomials of degree $p^2$ whose Galois group is a $p$-group}

In the proof of Theorem \ref{theo1} we obtained a list of requirements on an Eisenstein equation of
degree $p^2$ that guarantee that the generated extension is Galois cyclic over $K$. Let's keep the
hypotheses on the ramification numbers (and consequently conditions \ref{tcr1} and \ref{tcr2} of the
Theorem), it is a natural question to describe the Galois group of the normal closure when some of
these hypotheses are not satisfied.

Put $L=K(\pi)$, we will also keep the condition \ref{tcgal}, which we can see immediately to be
satisfied if and only if $L$ contains a Galois extension of degree $p$ of $K$, which is a necessary
condition for the normal closure of $L/K$ to be a $p$-group. Note that this hypothesis is always
satisfied for $f(X)$ if $K$ is replaced with a suitable unramified extension.

We can notice from the proof that the first unsatisfied requirement among the conditions
\ref{tcpp1}, \ref{tcell} with $\ell$ as big as possible, \ref{tc2} and \ref{tc1} in Theorem
\ref{theo1} gives us information about the biggest possible $\ell$ such that
$N_{L/K}(U_{\ell,L})\cap{}U_{2,K}$ is not contained in $1+p^2V$, with $V$ defined as in the
proof. We expect this fact to allow us to deduce information about the Galois group of the normal
closure.

Let $F$ be the Galois extension of degree $p$ of $K$ contained in $L$, it is unique or $\Gal(L/K)$
would be elementary abelian, which is not possible considering the ramification breaks. Then $F/K$
has ramification number $1$ and $L/F$ ramification number $p+1$. 

Before continuing we prove a proposition that will also be of use later, for which we only assume
$L/F$ to have ramification break $>1$ and $F/K$ to be Galois with break at $1$, with group generated
by $\sigma$ say. Similarly we have that the only $f_i$ with $v_p(f_i)=1$ are $f_p$ and
$f_{p^2}$. For some $\theta\in{}K$ we can write
\begin{align*}
   \pi_F^{(\sigma-1)} &= 1-\theta^p\pi_F+\cO(\pi_F^2)\\
     &= N_{L/F}(1-\theta\pi)+\cO(\pi_F^2)
 \end{align*}
in view of \cite[Chap. 3, \S1, Prop. 1.5]{fesenko2002local} and having $L/F$ ramification break
 $>1$. Since $\pi_F^{(\sigma-1)}$ is killed by $N_{F/K}$ and $N_{F/K}(U_{2,F})\subseteq{}U_{2,K}$
 we should have $N_{L/K}(1-\theta\pi)\in{}U_{2,K}$ and hence
 $\bar\theta^{p(p-1)}=-\sfr{F_p}{F_{p^2}}$, as we have expanding
 $\tilde{f}(\theta)=N_{L/F}(1-\theta\pi)$ like in the proof of Theorem \ref{theo1}.

We obtain inductively that
\begin{prop}
\label{prop_shift}
For each $1\leq\ell<p$ we have 
\[
   \pi_F^{(\sigma-1)^\ell}=1-k\theta^{p\ell}\pi_F^\ell+\cO(\pi_F^{\ell+1}),
\]
for some integer $k$ prime with $p$, where $\bar\theta^{p(p-1)}=-\sfr{F_p}{F_{p^2}}$.
\end{prop}

We return to our main problem, so let $L/F$ have ramification break at $p+1$.  We require $L/F$ to
be Galois: by local class field theory this is the case precisely when the map
$\sfr{U_{p+1,L}}{U_{p+2,L}}\rightarrow\sfr{U_{p+1,F}}{U_{p+2,F}}$ induced by $N_{L/F}$ is not
surjective. Since the map $\sfr{U_{p+1,F}}{U_{p+2,F}}\rightarrow\sfr{U_{2,K}}{U_{3,K}}$ induced by
$N_{F/K}$ is an isomorphism by \cite[Chap. 3, \S1, Prop. 1.5]{fesenko2002local}, we are reduced to
study the image of $\sfr{U_{p+1,L}}{U_{p+2,L}}$ in $\sfr{U_{2,K}}{U_{3,K}}$. Considering the norms
of the usual $1+\theta\pi^{p+1}$, we have from the proof of Theorem \ref{theo1} that this map is
described by the additive polynomial $A_{p+1}(Y)$, and is non-surjective precisely when
$\sfr{G_{p+1}}{F_pF_{p^2}}$ is in $\kk_K^{p-1}$. Consequently we will always assume the

\begin{condition}
  We require $\sfr{G_{p+1}}{F_pF_{p^2}}\in\kk_K^{p-1}$.
\end{condition}

\noindent which is necessary and sufficient for the Galois closure of $L/K$ to be a $p$-group, and again is
always satisfied if we replace $K$ by a suitable unramified extension.

For an $\F_p[G]$-module $M$ we respectively denote by $\soc^iM$ and $\rad^iM$ the $i$-th socle and
radical of $M$. If $\sigma$ is a generator of $G$, the radical of $\F_p[G]$ is generated by
$\sigma-1$, and we have
\[
   \rad^i M = M^{(\sigma-1)^i},\qquad \soc^i M = \left\{ x : x^{(\sigma-1)^i} = 0 \right\}.
   \]

Let $G=\Gal(F/K)$ and $\tilde{L}$ be the Galois closure of $L$ over $K$, we want to compute the
length of $\Gal(\tilde{L}/F)$ as a $\F_p[G]$-module, which we will also show to determine completely
$\Gal(\tilde{L}/K)$ in the present case. If $F^{(p)}$ is the maximal abelian elementary
$p$-extension of $F$, this amounts to computing the smallest $m$ such that $rad^m\Gal(F^{(p)}/F)$ is
contained in $\Gal(F^{(p)}/L)$.

For $0\leq{}i\leq{}p$ let's consider the submodules $S_i=\soc^{p-i}P_F$ of $P_F=[F^\times]_F$ (which
is canonically identified with $\Gal(F^{(p)}/F)$ via local class field theory), and let $K_i$ be the
class field corresponding to $S_i$ over $F$. For $0\leq{}i<p$ we have $[U_{i+1,F}]\subseteq{}S_i$
and thus the highest upper ramification break of $\Gal(K_i/F)$ is $i$ for $i<p$, and in particular
being $p+1$ the unique ramification break of $L/F$ we have that $K_i\nsupseteq{}L$ for $i<p$. Note
also that $K_1$ is the maximal elementary abelian $p$-extension of $K$.

Let $K'$ be the field corresponding to $\rad^1P_F$, it is the maximal $p$-elementary abelian
extension of $F$ which is abelian over $K$, and it corresponds to $N_{F/K}(F^\times)^p$ via the
class field theory of $K$. Considering the structure of $P_F\cong\F_p[G]^{\oplus{}f}\oplus\F_p$ as a
Galois module we have that
\[
     rad^iP_F = \soc^{p-i}P_F \cap \rad^1P_F = S_i\cap \rad^1P_F,
\]
for each $i$, and $rad^iP_F$ corresponds to $K'K_i$ via class field theory, so we are looking for
the smallest $m$ such that $L\subset{}K'K_m$. Since $L$ and $K'$ are never contained in $K_i$ for
$i<p$ and $K'$ has degree $p$ over $K_1$, this inclusion holds if and only if $L$ and $K'$ generate
the same extension over $K_m$ (and $\tilde{L}$ will too, being $K'K_m$ Galois over $K$). This is the
case if and only if $K'\subset{}LK_m$, and this condition is consequently equivalent to the
$\F_p[G]$-module $\Gal(\tilde{L}/F)$ having length $\leq{}m$

We can now show that if $K'\subset{}LK_m$ for some $m<p$, then $\Gal(\tilde{L}/K)$ cannot be the
semidirect product: indeed $\Gal(K_{m+1}/K)$ lives in the exact sequence
\[
    1 \rightarrow \sfr{P_F}{S_{m+1}} \rightarrow \Gal(K_{m+1}/K) \rightarrow G \rightarrow 1,
\]
and all $p$-th powers in $\Gal(K_{m+1}/K)$ are clearly $G$-invariant elements of
$\sfr{P_F}{S_{m+1}}$, and hence contained in $\sfr{S_m}{S_{m+1}}$, and this shows that the quotient
$\Gal(K_m/K)$ has exponent $p$ since we quotiented out all $p$-th powers. On the other hand
$\Gal(K'/K)$ has exponent $p^2$ so if $K'\subset{}LK_m$ then also $\Gal(\tilde{L}K_m/K)$ does, and
$\Gal(\tilde{L}/K)$ should also have exponent $p^2$ or the $p$-th power of any element of the
absolute Galois group would act trivially on $\tilde{L}$, $K_m$, and consequently on $\tilde{L}K_m$,
which is impossible. Note that if $\Gal(\tilde{L}/F)$ has maximal length $m=p$ there is only one
possible isomorphism class of possible $p$-groups, which is the wreath product of two cyclic groups
of order $p$, see \cite{waterhouse1994normal,minac2005galois} for more information about these
groups.

The above observation can be viewed as the fact that, for $m<p$, $K_m$ is the compositum of all the
extensions of degree $p$ whose normal closure has group over $F$ of length $\leq{}m$ as
$\F_p[G]$-module, and whose group over $K$ is the semidirect product extension (and hence has
exponent $p$). The extensions whose group of the normal closure over $K$ is not the semidirect
product are obtained via a sort of twist with $K'$, which is non-trivial when $L\nsubseteq{}K_m$.

Now $K'$ is not contained in $LK_m$ precisely when there exist an element in
$\Gal(K^{\mathrm{ab}}/K)$ fixing $LK_m$ but not $K'$, any such element can be lifted to
$\Gal(L^{\mathrm{ab}}/L)$. Since the image of the Artin map
$\Psi_L:L^\times\rightarrow\Gal(L^{\mathrm{ab}}/L)$ is dense in $\Gal(L^{\mathrm{ab}}/L)$ we can
take such element of the form $\Psi_L(\alpha)$ for some $\alpha\in{}L^\times$. Having to fix $K_1$
we will have $N_{L/K}(\alpha)\in{}(K^\times)^p$ by the functoriality of the reciprocity map (see
\cite[Chap. 4, Theorem 4.2]{fesenko2002local}), $[N_{L/F}(\alpha)]_F\in{}S_m$ because $K_m$ is
fixed, and $N_{L/K}(\alpha)\notin{}N_{F/K}(F^\times)^p$ because the action is non-trivial on $K'$,
and on the other hand the existence of such an element ensures that $K'\notin{}LK_m$.

If $L$ and $K$ are as above, we have proved the

\begin{prop}
\label{clos1}
Let $1\leq{}m\leq{}p$ be the smallest possible integer such that for all $\alpha\in{}L^\times$ such
that $N_{L/K}(\alpha)\in(K^\times)^p$ and $[N_{L/F}(\alpha)]_F\in{}S_m$ we also have
$N_{L/K}(\alpha)\in{}N_{F/K}(F^\times)^p$. Then $\Gal(\tilde{L}/K)$ is the unique $p$-group which
has exponent $p^2$ and is an extension of $G=\Gal(F/K)$ by an indecomposable $\F_p[G]$-module of
length $m$.
\end{prop}

We now determine the $(p-m)$-th socle $S_m$ of $P_F$ for each $0\leq{}m\leq{}p$, and deduce the
ramification jumps of the normal closure.

Consider the images $V_i=[U_{i,F}]_F$ of the $U_{i,F}$ in $P_F$ for $i\geq{}1$, and put $V_0=P_F$
for convenience. If $G$ is generated by $\sigma$ say, the radical of $\F_p[G]$ is generated by
$(\sigma-1)$ and we have $V_i^{\sigma-1}\subseteq{}V_{i+1}$. Since $V_p=V_{p+1}$ and $V_{p+2}=1$ we
have that $V_p$ is killed by $\sigma-1$, $V_{p-1}$ by $(\sigma-1)^2$ and so on, so that
$V_{k+1}\subseteq{}\soc^{p-k}P_F=S_{k}$ for $0\leq{}k<p$, while clearly $S_p=0$. Furthermore if
$\pi_F$ is a uniformizing element of $F$ we have $\pi_F^{(\sigma-1)^k}\in{}V_k\setminus{}V_{k+1}$
and $\pi_F^{(\sigma-1)^{k}}\in{}S_k$ for $0\leq{}k<p$, so comparing the dimensions we have that
\[
   S_k = \langle\pi_F^{(\sigma-1)^k}\rangle + V_{k+1}.
\]

If $m$ is like in the proposition and $\geq2$, take in $L^\times$ an element $\alpha$ contradicting
the proposition for $m-1$ and such that $t=v_L(1-N_{L/F}(\alpha))$ is as big as possible. Then
$\psi_{LK_{m-1}/F}(t)$ is the ramification break of $K'LK_{m-1}/LK_{m-1}$, which is also equal to
that of $LK'K_{m-1}/K'K_{m-1}$ considering that $K'K_{m-1}/K_{m-1}$ and $LK_{m-1}/K_{m-1}$ have the
same ramification break equal to $\psi_{K_{m-1}/F}(p+1)$, and the total set of breaks has to be
preserved. By the definition of $S_{m-1}$ and $S_m$ we have that $t$ can be either $m-1$ or $m$,
unless $m=p$ where $t$ is either $p-1$ or $p+1$.

By local class field theory $K'K_{m-1}/F$ corresponds to the subgroup $A=rad^{m-1}P_F$ of $P_F$, and
$LK'K_{m-1}/F$ to another subgroup $B$ with index $p$ in $A$, and $t$ is the biggest $t$ such that
some $x\in{}V_t\cap{}A$ has non-trivial image in $A/B$. Passing to the groups $A'$ and $B'$ of the
elements sent by $\sigma-1$ into $A$ and $B$ respectively, $A'=soc^{p-m+2}P_F$ corresponds to
$K_{m-2}$, and $B'$ to $L'K_{m-2}$ where $L'$ is the subfield of $\tilde{L}$ corresponding to
$\soc^1\Gal(\tilde{L}/F)$ as $\F_p[G]$-module. The upper ramification break of the new relative
extension is $\psi_{K_{m-2}/F}(s)$ where $s$ is the biggest so that some $y\in{}V_s\cap{}A'$ is
nontrivial in $A'/B'$. Being $A=rad^{m-1}P_F$ each $x\in{}A\setminus{}B$ is of the form
$x=y^{\sigma-1}$ for some $y\in{}A'\setminus{}B'$, so $s=t-1$ unless $t=p+1$ which becomes $s=p-1$.

Since $\Gal(L'/F)$ has length $m-1$ and the field $L''$ corresponding to $\soc^1\Gal(L'/L)$ is
contained in $K_{m-2}$, and $V_{m-2}\supseteq{}A'\supseteq{}V_{m-1}$, we have that $s$ is also the
ramification number of $L'/L''$. Repeating this observation for $m-1$ steps we have that the
ramification breaks over $F$ are either $1,2,\dots,m-1,p+1$, either $0,1,\dots,m-2,p+1$ depending on
whether an element $\alpha\in{}S_{m-1}$ contradicting the proposition can be found in $V_m$ or not,
where for convenience a ``ramification break'' of $0$ indicates an unramified extension.

We proved the
\begin{prop}
\label{clos2}
Let $1\leq{}m\leq{}p$ be like in the Prop. \ref{clos1}, if we can find an $\alpha$ such that
$N_{L/K}(\alpha)\in{}(K^\times)^p\setminus{}N_{F/K}(F^\times)^p$ such that
$[N_{L/F}(\alpha)]_F\in{}V_m$, then the normal closure $\tilde{L}/F$ is totally ramified with breaks
$1,2,\dots,m-1,p+1$. If not, then $\tilde{L}/F$ is formed by an unramified extension of degree $p$
and an extension with breaks $1,2,\dots,m-2,p+1$.
\end{prop}

We will look for the biggest $1\leq{}m\leq{}p-1$ such that we can find an $\alpha$ contradicting the
requests of the Prop. \ref{clos1}. For all $\ell=p-1,\dots,2,1$ in descending order, if we cannot
find a suitable $\alpha$ with $[N_{L/F}(\alpha)]_F\in{}V_{\ell+1}$, we inductively test
$S_{\ell}\supset{}V_{\ell+1}$ (and $\Gal(\tilde{L}/F)$ has length $\ell+1$ and there is an
unramified part), and then $V_{\ell}\supseteq{}S_\ell$ (and in this case $\Gal(\tilde{L}/F)$ has
length $\ell$ and the extension is totally ramified).

Verifying that we cannot find $\alpha$ with $[N_{L/F}(\alpha)]_F\in{}V_{\ell+1}$ is easy, and is the
condition of the theorem connected to $A_{p+1}(Y)$ for $\ell=p-1$, or to $A_{\ell+1}$ if
$\ell<p-1$. We then allow $[N_{L/F}(\alpha)]_F$ to be in
$S_\ell=\langle\pi_F^{(\sigma-1)^{\ell}}\rangle+V_{\ell+1}$: by Prop. \ref{prop_shift} for
$\bar\theta^{p(p-1)}=-\sfr{F_p}{F_{p^2}}$ and for some $k$ prime with $p$ we have
\begin{align*}
  \pi_F^{(\sigma-1)^\ell} &= 1-k\theta^{p\ell}\pi_F^\ell+\cO(\pi_F^{\ell+1})\\
  &= N_{L/F}(1-k\theta^\ell\pi^\ell)+\cO(\pi_F^{\ell+1}),
\end{align*}
in view of \cite[Chap. 3, \S1, Prop. 1.5]{fesenko2002local} and being $\ell$ smaller than the
ramification number $p+1$.  In particular the image of $N_{L/F}(1-\theta^\ell\pi^\ell)$ generates
$\sfr{S_\ell}{V_{\ell+1}}$, and testing the condition for $S_\ell$ is equivalent to verifying that
$A_{\ell}(\bar\theta^\ell)\in{}V$.

Note that $A_2(\bar\theta^2)$ has the simplified form $G_{2p}\bar\theta^{2p}+G_{2}\bar\theta^2$, and
testing if $F_{p^2}X^p+F_pX=A_{\ell}(\bar\theta^\ell)$ has solution in $\kk_K$ is equivalent to
checking, after replacing $X$ by $\bar\theta^\ell{}X$ and dividing by $\bar\theta^\ell$, if there
are solutions to
\[
         F_{p^2}(-\sfr{F_p}{F_{p^2}})^{\sfr{\ell}{p}}X^p + F_pX -
         G_{p\ell}(-\sfr{F_p}{F_{p^2}})^{\sfr{\ell}{p}} - G_\ell = 0.
\]
Note that for $\ell=1$ we just test if $\overline{\sfr{c_1(\theta)}{p^2}}$ is in $V$, like in the
last condition of Theorem \ref{theo1}.

We have the

\begin{theo}
\label{theo2}
Assume $f(X)$ to satisfy conditions \ref{tcr1}, \ref{tcr2}, \ref{tcgal} of Theorem \ref{theo1}, and
keeping the notation assume additionally that
\begin{enumerate}[leftmargin=2.5em]
\item $\sfr{G_{p+1}}{F_pF_{p^2}}\in\kk_K^{p-1}$.
\saveenum
\end{enumerate}
Let $L$ be the extension determined by $f(X)$, $\tilde{L}$ the normal closure over $K$, and $F$ the
unique subextension of degree $p$ contained in $L$. Then $\Gal(\tilde{L}/K)$ is an extension of
$G=\Gal(F/K)$ by the indecomposable $\F_p[G]$-module $M=\Gal(\tilde{L}/F)$, $\Gal(\tilde{L}/K)$ has
exponent $p^2$ and is a non-split extension unless $M$ has length $p$. Furthermore
\begin{enumerate}[leftmargin=2.5em]
\loadenum
\item if $G_{p+1}^p\neq-F_p^{p+1}$ then $M$ has length $p$ and $L/F$ is totally ramified with upper
  ramification breaks $1,2,\dots,p-1,p+1$;
\saveenum
\end{enumerate}
assuming equality in the previous condition, 
\begin{enumerate}[leftmargin=2.5em]
\loadenum
\item if $G_{p\ell}\neq{}F_{p^2}\left(\sfr{G_\ell}{F_p}\right)^p$ for some $\ell\in\I{3,p-1}$ that
  we take as big as possible, or
  $G_{2p}\neq{}F_{p^2}\left(\sfr{G_2}{F_p}\right)^p+\frac{1}{2}F_p\left(F_p-F_{p^2}^{\sfr{1}{p}}\right)$
  and we put $\ell=2$, let 
\[
U(X) = F_{p^2}(-\sfr{F_p}{F_{p^2}})^{\sfr{\ell}{p}}X^p + F_pX -
G_{p\ell}(-\sfr{F_p}{F_{p^2}})^{\sfr{\ell}{p}} - G_\ell.
\]
We have that
\begin{itemize}[leftmargin=1.5em]
\item if $U(X)$ has no root in $\kk_K$, then $M$ has length $\ell+1$ and $\tilde{L}/F$ is formed by an
  unramified extension followed by a totally ramified extension with upper ramification breaks
  $1,2,\dots,\ell-1,p+1$,
\item if $U(X)$ has some root in $\kk_K$, then $M$ has length $\ell$ and $\tilde{L}/F$ is a totally
  ramified extension with upper ramification breaks $1,2,\dots,\ell-1,p+1$,
\end{itemize}
\saveenum
\end{enumerate}
assuming equality in the previous conditions, and for $\bar\theta^{p(p-1)}=-\sfr{F_p}{F_{p^2}}$,
\begin{enumerate}[leftmargin=2.5em]
\loadenum
\item if
  $F_{p^2}X^p+F_pX-\overline{\frac{1}{p^2}\left(f_{p^2}\theta^{p^2}+f_{p}\theta^{p}\right)}-G_1\bar\theta$ has
  no root in $\kk_K$, then $M$ has length $2$ and $\tilde{L}/F$ is formed by an unramified extension
  followed by a totally ramified extension with upper ramification break $p+1$.
\end{enumerate}
All conditions pass precisely when all requirements of Theorem \ref{theo1} are satisfied, and in
this case $L/F$ is Galois cyclic.
\end{theo}

It turns out that we just worked out the hard case of the classification of all polynomials of
degree $p^2$ whose Galois group is a $p$-group.

We keep the notation of the previous part of this section. We have classified in Theorem \ref{theo2}
all polynomials such that $L/F$ has ramification break at $p+1$ and the normal closure is a
$p$-group, and it turned out that the condition on the ramification number is sufficient to
guarantee that the Galois group of the normal closure has exponent $p^2$. Conversely if the
ramification number is $\leq{}p-1$ then either $L\subset{}K_m$ for some $m<p$ and
$\Gal(\tilde{L}/F)$ has length $\leq{}m$ and $\Gal(\tilde{L}/K)$ is the splitting extension of $G$,
either $\Gal(\tilde{L}/F)$ has length $p$ and there is only one possibility for $\Gal(\tilde{L}/K)$
which is both a split extension and has exponent $p^2$, and is a wreath product of two cyclic groups
of order $p$.

Again as above, assume $m$ to be the smallest integer such that $[N_{L/F}(L^\times)]_F$ contains
$V_{m+1}$. The ramification number of $L/F$ is $m$, and the length of $\Gal(\tilde{L}/K)$ as
$G$-module can be $m$ when the norms also contain $S_m$, or $m+1$ if this is not the case. Since
$S_\ell=\langle\pi_F^{(\sigma-1)^{\ell}}\rangle+V_{\ell+1}$ to resolve this ambiguity we should test
whether $[\pi_F^{(\sigma-1)^{\ell}}]_F\in{}[N_{L/F}(L^\times)]_F$. Since
$\pi_F^{(\sigma-1)^{\ell}}\in{}U_{\ell,F}$ and $N_{L/F}(L^\times)\supset{}U_{\ell+1,F}$ we can just test if
\[
    N_{L/F}(1+\theta\pi^{\ell}) = \pi_F^{(\sigma-1)^{\ell}} + \cO(\pi^{\ell+1})
\]
for some unit $\theta\in{}K$.

Factorizing in $L$ the ramification polynomial $f(X+\pi)$ over the Newton polygon we have that
$f(X+\pi)=Xg(X)h(X)$, where $g(X)$ has degree $p-1$ with roots of valuation $\ell+1$ and $h(X)$
degree $p^2-p$ and roots with valuation $2$. We can take $g(X)$ to be monic and with roots
$\tau^i(\pi)-\pi$, where $\tau$ is an automorphism of order $p$ of the normal closure of $L$ over
$F$ and $1\leq{}i<p$, and $L/F$ is Galois if and only if $g(X)=X^{p-1}+\dots+g_1X+g_0$ splits in
linear factors in $L$. If we can write $\tau(\pi)-\pi=\eta\pi^{\ell+1}+\cO(\pi^{\ell+2})$ with
$\eta\in{}K$, then $\tau(\pi)$ can be approximated in $L$ better than by any other conjugate, and
consequently $L/F$ is Galois by Krasner lemma. On the other hand if $L/K$ is Galois we certainly
have such an expression for some $\eta$. Since
\begin{align}
   g_0&=\prod_{i=1}^{p-1}(\tau^i\pi-\pi)\\
     &\equiv\prod_{i=1}^{p-1}i\eta\pi^{\ell+1}\equiv -\eta^{p-1}\pi^{(p-1)(\ell+1)}
   \end{align}
we have that $L/F$ is Galois if and only if $-g_0$ is a $(p-1)$-th power.

The monomial in $X^p$ of $f(X+\pi)$ is
\[
   \binom{p^2-p}{p}f_p\pi^{p^2-2p}X^p=h_0X^p
\]
where $h_0$ is the constant term of $h(X)$, while the monomial in $X$ is
\[
   Xf'(\pi) = (p^2-r)f_r\pi^{p^2-r-1}X = g_0h_0X
\]
where $r$ should be $p^2-(p-1)\ell+p$ and $v_p(f_{p^2-r})=2$, considering that $v_L(f'(\pi))$ is
$(p^2-p)\cdot2+(p-1)\cdot(\ell+1)$.
   
Since $\binom{p^2-p}{p}\equiv-1\imod{p}$ by the definition of $r$ we have taking the ratio of the
coefficients of the monomials above that
\begin{align*}
  \frac{g_0}{\pi^{(p-1)(\ell+1)}} &=
  \frac{-rf_r\pi^{(p-1)\ell-p-1}}{-f_p\pi^{p^2-2p}}\cdot\pi^{-(p-1)(\ell+1)} + \cO(\pi)\\
  &= \sfr{rf_r}{f_p} \cdot\pi^{-p^2} + \cO(\pi) = -\sfr{rf_r}{f_pf_{p^2}} + \cO(\pi),
\end{align*}
being $\pi^{p^2}=-f_0+\cO(\pi)$.

Since $r\equiv\ell\imod{p}$ we obtained that $\bar\eta^{p-1}$ is equal to
$\overline{\sfr{\ell{}f_r}{f_pf_{p^2}}}$, and it is contained in $\kk_K^{p-1}$ if and only if $g_0$
is a $p-1$-th power. Put again $F_p=\overline{\sfr{f_p}{p}}$, $F_{p^2}=\overline{\sfr{F_{p^2}}{p}}$
and $G_i=\overline{\sfr{f_i}{p^2}}$ for $i\neq{}p,p^2$.

\begin{condition}
  $L/F$ is Galois if and only if $\sfr{\ell{}G_r}{F_pF_{p^2}}$ is in $\kk_K^{p-1}$, where $r$ is
  equal to $p^2-(p-1)\ell+p$.
\end{condition}

Let's recall that from \cite[Chap. 3, \S1, Prop. 1.5]{fesenko2002local} we have that
\[
    N_{L/F}(1+\theta\pi^\ell) = 1+(\theta^p-\eta^{p-1}\theta)\pi_F^\ell+\cO(\pi_F^{\ell+1}),
    \]
while
\[
   \pi_F^{(\sigma-1)^\ell}=1-k\rho^\ell\pi_F^{\ell}+\cO(\pi_F^{\ell+1})
\]
for $\bar\rho^{p-1}=-\sfr{F_p}{F_{p^2}}$ and some integer $k$ prime with $p$, by
Prop. \ref{prop_shift}. From what observed at the beginning, we obtain that the length of
$\Gal(\tilde{L}/K)$ is $\ell$ when $X^p-\sfr{\ell{}G_r}{F_pF_{p^2}}X=\bar\rho^\ell$ has solution in
$\kk_K$, and $\ell+1$ if this is not the case. Replacing $X$ by $\bar\rho^\ell{}X$ and dividing by
$\bar\rho^\ell$ this is equivalent to testing if
\[
   (-\sfr{F_p}{F_{p^2}})^\ell{}X^p - \sfr{\ell{}G_r}{F_pF_{p^2}}X = 1
\]
has solution in $\kk_K$.

Consequently we obtain

\begin{theo}
\label{theo3}
Let $2\leq\ell\leq{}p-1$ an let $r=p^2-(p-1)\ell+p$, and assume that $f(X)$ is such that
\begin{enumerate}[leftmargin=2.5em]
\item $v_p(f_{p})=1$, and $v_p(f_{pi})\geq2$ for $i\in\I{2,p-1}$,
\item $v_p(f_{i})\geq2$ for $i\in\I{1,r-1}$, $v(f_r)=2$ and $v_p(f_{i})\geq3$ for
  $i\in\I{r+1,p^2-1}$,
\saveenum
\end{enumerate}
putting $F_p=\overline{\sfr{f_p}{p}}$, $F_{p^2}=\overline{\sfr{f_{p^2}}{p}}$, and
$G_i=\overline{\sfr{f_i}{p^2}}$ for all $i\neq{}p,p^2$ we have
\begin{enumerate}[leftmargin=2.5em]
\loadenum
\item $-\sfr{F_p}{F_{p^2}}=\bar\rho^{p-1}$ for some $\bar\rho\in\kk_K^\times$,
\item $\sfr{\ell{}G_r}{F_pF_{p^2}}=\bar\eta^{p-1}$ for some $\bar\eta\in\kk_K^\times$.
\end{enumerate}
Let $L$ be the extension determined by $f(X)$, $\tilde{L}$ the normal closure over $K$, and $F$ the
unique subextension of degree $p$ contained in $L$. Then $\Gal(\tilde{L}/K)$ is a split extension of
$G=\Gal(F/K)$ by the indecomposable $\F_p[G]$-module $M=\Gal(\tilde{L}/F)$ and furthermore defining
\[
  U(X) = (-\sfr{F_p}{F_{p^2}})^\ell{}X^p - \sfr{\ell{}G_r}{F_pF_{p^2}}X - 1
\]
we have that if
\begin{itemize}[leftmargin=1.5em]
\item $U(X)$ has no root in $\kk_K$, then $M$ has length $\ell+1$ and $\tilde{L}/K$ is formed by an
  unramified extension followed by a totally ramified with upper ramification breaks
  $1,2,\dots,\ell$,
\item $U(X)$ has some root in $\kk_K$, then $M$ has length $\ell$ and $\tilde{L}/K$ is totally
  ramified with upper ramification breaks $1,2,\dots,\ell$.
\end{itemize}
\end{theo}

What is left is the easy case for $\ell=1$, which is considered separately. In this case $L/K$ has
$1$ as unique ramification jump, $v_p(f_1)=1$ while $v_p(f_i)\geq2$ for $i\in\I{2,p^2-1}$, and
consequently put $F_i=\overline{\sfr{f_i}{p}}$ for $i=1,p,p^2$. The map
$\sfr{U_{1,L}}{U_{2,L}}\rightarrow\sfr{U_{1,K}}{U_{2,K}}$ induced by $N_{L/K}$ is described by the
additive polynomial $A(Y)=F_{p^2}Y^{p^2}+F_pY^p+F_1Y$, and $L/K$ is Galois precisely when
$N_{L/K}(U_{1,L})=1+pW$ for a subspace $W$ of codimension $2$ in $\kk_K$, that is when $A(Y)$ splits
completely in $\kk_K$. On the other hand the normal closure $\tilde{L}/K$ is a $p$-extension if and
only if $L$ becomes abelian elementary over the unique unramified extension of degree $p$ of $K$, or
equivalently if $A(Y)$ splits completely over the unique extension of degree $p$ of $\kk_K$.

\begin{theo}
\label{theo4}
Assume that $f(X)$ is such that
\begin{enumerate}[leftmargin=2.5em]
\item $v_p(f_{p})\leq{}1$, and $v_p(f_{pi})\geq2$ for $i\in\I{2,p-1}$,
\item $v_p(f_1)=1$, and $v(f_i)\geq2$ for $i\in\I{2,p^2-1}$,
\saveenum
\end{enumerate}
and putting $F_i=\overline{\sfr{f_i}{p}}$ for $i=1,p,p^2$
\begin{enumerate}[leftmargin=2.5em]
\loadenum
\item the polynomial $F_{p^2}Y^{p^2}+F_pY^p+F_1Y$ has a root in $\kk_K^\times$, and
  $\sfr{F_1}{F_{p^2}}\in(\kk_K^\times)^{p-1}$.
\end{enumerate}
Let $L$ be the extension determined by $f(X)$, $\tilde{L}$ the normal closure over $K$, and $F$ the
unique subextension of degree $p$ contained in $L$.  Then
\begin{itemize}[leftmargin=1.5em]
\item if $F_{p^2}Y^{p^2}+F_pY^p+F_1Y$ does not split in $\kk_K$ then $M$ has length $2$, and
  $\tilde{L}/F$ is formed by an unramified extension followed by a totally ramified extension with
  upper ramification break $1$, and $\Gal(\tilde{L}/K)$ is a split extension of $\Gal(F/K)$ by
  $\Gal(\tilde{L}/F)$,
\item if $F_{p^2}Y^{p^2}+F_pY^p+F_1Y$ has all roots in $\kk_K$ then $M$ has length $1$ and $L/K$ is
  an abelian elementary $p$-extension.
\end{itemize}
\end{theo}

Theorems \ref{theo2}, \ref{theo3} and \ref{theo4} cover all possible ramification jumps of the
extension $L/F$, so they completely describe the Galois groups of polynomials of degree $p^2$ whose
splitting field is a $p$-extension.

\section{Polynomials of degree $p^3$ generating a cyclic extension}
We proceed with the same strategy used for the polynomials of degree $p^2$, starting from the
conditions on the valuations of the coefficients.

Let $f(X)=X^{p^3}+\dots+f_{p^3-1}X+f_{p^3}$, since the different has now valuation $4p^3-p^2-p-2$ it
will be determined by the monomial $f_{p^2+p+1}X^{p^3-p^2-p-1}$, $v_p(f_{p^2+p+1})=3$,
$v_p(f_i)\geq3$ if $(i,p)=1$ and $v_p(f_i)\geq4$ if furthermore $i>p^2+p+1$. Let $\pi$ be a root,
the coefficients of the term of degree $p$ of the ramification polynomial $f(X+\pi)$ will have
valuation $(p^3-p^2)\cdot2+(p^2-p)\cdot{}(p+1)=3p^3-p^2-2p$ and has to come from a monomial
$f_{p^3-i}(X+\pi)^i$ contributing the term $\binom{i}{p}f_{p^3-i}X^p\pi^{i-p}$, and we deduce the
$i$ has to be $i=p^3-p^2-p$, that $v_p(f_{p^2+p})=2$, that $v_p(f_{pi})\geq2$ for $(i,p)=1$ and
$v_p(f_{pi})\geq3$ if furthermore $i\geq{}p+2$. Similarly considering the coefficient of the term of
degree $p^2$ of the ramification polygon which shall have valuation $\geq2p^3-2p^2$ we obtain that
$v_p(f_{p^2})=1$ and $v_p(f_{p^2i})\geq2$ for all indices such that $(i,p)=1$.

\begin{condition}
We must have
\begin{enumerate}
\item $v_p(f_{p^2})=1$ and $v_p(f_{p^2i})\geq{}2$ for $i\in\I{2,p-1}$,
\item $v_p(f_{pi})\geq{}2$ for all $i\in\I{1,p-1}$, $v_p(f_{p^2+p})=2$, and $v_p(f_{pi})\geq{}3$
  for all $i\in\I{p+1,p^2-1}$,
\item $v_p(f_{i})\geq{}3$ for all $i\in\I{1,p^2+p-1}$, $v_p(f_{p^2+p+1})=3$ and
  $v_p(f_{i})\geq{}4$ for all $i\in\I{p^2+p+2,p^3-1}$.
\end{enumerate}
\end{condition}

Again working like in degree $p^2$, we shall require $N_{L/K}(U_{1,L})^{p^{i-1}}\cap{}U_{i+1,L}$ to
be contained in $1+p^iV$ for $1\leq{}i\leq3$ and some $\F_p$-vector space $V$, and after determining
$V$ we will have to verify the condition on the combinations of the norms of elements of the form
$1+\theta\pi^{\ell}$ for a unit $\theta$, and $1\leq\ell\leq{}p^2+p+1$ and $(\ell,p)=1$.

Let's expand again $\prod_{i=0}^\ell{}\tilde{f}(\zeta_\ell^i{}T)$ modulo $p^4$, taking into account
the valuations of the $f_i$ and evaluating directly the $\Sigma_\lambda(\ell)$ via
Prop. \ref{prop_sums} it can be written with the terms in increasing valuation as

\begin{align}
  \kp_K\!\ni\! &\left[\qquad\qquad\qquad\qquad\qquad
\begin{gathered}
  + \delta_{\ell,1}f_{p^2}T^{p^2} + \delta_{\ell,1}f_{p^3}T^{p^3}
\end{gathered}
\right.\\
  \kp_K^2\!\ni\! &\left[\qquad\qquad
\begin{gathered}
  + \frac{1}{2}\delta_{\ell,2}^{[2]}f_{p^2}^2T^{2p^2} + \delta_{\ell,p+1}^{[2]}f_{p^2}f_{p^3}T^{p^3+p^2}+
  \frac{1}{2}\delta_{\ell,2}^{[2]}f_{p^3}^2T^{2p^3} \\
  +  \sum_{j\in\I{2,p-1}} \delta_{\ell,j}f_{p^2j}T^{p^2j}
  +  \sum_{k\in\I{1,p+1}} \delta_{\ell,k}f_{pk}T^{pk}
\end{gathered}
\right.\\
  \kp_K^3\!\ni\! &\left[
\begin{gathered}
  + \frac{1}{3}\delta_{\ell,3}^{[3]}f_{p^2}^3T^{3p^2}
  + \delta_{\ell,p+2}^{[3]}f_{p^3}f_{p^2}^2T^{p^3+2p^2}
  + \delta_{\ell,2p+1}^{[3]}f_{p^3}^2f_{p^2}T^{2p^3+p^2}
  + \frac{1}{3}\delta_{\ell,3}^{[3]}f_{p^3}^3T^{3p^3}\\
  +  \!\!\!\!\sum_{j\in\I{2,p-2}}\! \!\delta_{\ell,j+1}^{[2]} f_{p^2}f_{p^2j}T^{p^2+p^2j}
  +  \delta_{\ell,1}^{[2]} f_{p^2}f_{p^3-p^2}T^{p^3}
  +  \!\!\!\!\sum_{k\in\I{1,p+1}}\!\! \delta_{\ell,k+p}^{[2]} f_{p^2}f_{pk}T^{p^2+pk}\\
  +  \sum_{j\in\I{2,p-1}} \delta_{\ell,j+p}^{[2]} f_{p^3}f_{p^2j}T^{p^3+p^2j}
  +  \sum_{k\in\I{1,p+1}} \delta_{\ell,k+p^2}^{[2]} f_{p^3}f_{pk}T^{p^3+pk}\\
  +  \sum_{j\in\I{p+2,p^2-1}} \delta_{\ell,j} f_{pj}T^{pj}
  +  \sum_{k\in\I{1,p^2+p+1}} \delta_{\ell,k} f_{k}T^{k}
\end{gathered}
\right.
\end{align}
While this expansion looks scary we can start noticing that since raising to a $p$-th power induces
an automorphism on the set of multiplicative representatives we have considering the expansion
modulo $p^3$ that the conditions stated in Theorem \ref{theo1} must be satisfied with $f_{pi}$in
place of $f_i$. Consequently put $F_{i}=\overline{\sfr{f_{i}}{p}}$ for $i=p^2,p^3$,
$G_i=\overline{\sfr{g_i}{p^2}}$ for $i\in{}p\I{1,p+1}$ or $i\in{}p^2\I{2,p-1}$, let
$A(Y)=F_{p^3}Y^p+F_{p^2}Y$ and put $V=A(\kk_K)$. Such conditions are satisfied if and only if $V$
has codimension $1$ in $\kk_K$ and the norms contained in $U_{1,K}$ or $U_{2,K}$ are respectively in
$1+pV$ and $1+p^2V$.

Similarly to the case in degree $p^2$, for $\ell\geq2$ this sum can be written as
$D_\ell(\theta)=\sum_{\ell\mid{}k}d_k(\theta^k)$ where the $d_\ell(T^\ell)$ are the polynomial
obtained if every $\delta^{[m]}_{\ell,i}$ is interpreted as a Kronecker's delta and
$1+d_\ell(\theta)\equiv{}N(E(\theta\pi^\ell))\imod{p^4}$. For $\ell=1$ there are exceptions because
$\delta^{[m]}_{\ell,i}=0$ for $\ell<m$.

We require the norms in $U_{3,K}$ to be in $1+p^3V$, and let's concentrate first on the case of
$\ell\in\I{p+2,p^2+p+1}$ so that the norms $N_{L/K}(1+\theta\pi^\ell)$ already live in $U_{3,K}$,
and the first few terms of the expansion disappear. For such indices $\ell$, $d_\ell(\theta)$ shall
be in $p^3V$ for each representative $\theta$, and dividing by $p^3$ we can consider the additive
polynomials $A_\ell(Y)=\overline{\sfr{d_\ell(Y)}{p^3}}$ which depending on $\ell$ are
\begin{gather*}
  \qquad \qquad\qquad    -G_{p(\ell-p^2)}F_{p^3}Y^p + H_{\ell}Y
       \qquad\qquad\qquad \qquad \ell\in\I{p^2+1,p^2+p+1},\\
 \qquad\qquad\qquad\qquad H_{p\ell}Y^p + H_{\ell}Y  
       \qquad\qquad\qquad\qquad\qquad \ell\in\I{2p+2,p^2-1},\\
 \qquad  F_{p^3}^2F_{p^2}Y^{p^2}+(H_{p\ell}-F_{p^2}G_{p(p+1)})Y^p + H_{\ell}Y 
       \quad\qquad \ell=2p+1,\\
 \qquad  -F_{p^3}F_{p^2(\ell-p)}Y^{p^2}+(H_{p\ell}-F_{p^2}G_{p(\ell-p)})Y^p + H_{\ell}Y
       \quad\qquad \ell\in\I{p+3,2p-1},\\
 (F_{p^3}F_{p^2}^2-F_{p^3}F_{2p^2})Y^{p^2}+(H_{p\ell}-F_{p^2}G_{2p})Y^p + H_{\ell}Y \qquad \ell=p+2,
\end{gather*}
where we have put $H_k=\overline{\sfr{f_k}{p^3}}$ for $k\in\I{1,p^2+p+1}$ and $k\in{}p\I{p+2,p^2-1}$.

\begin{condition}
For each $\ell\in\I{p+2,p^2+p+1}$ we shall have $A_\ell(\kk_K)\subseteq{}A(\kk_K)$.
\end{condition}

For $\ell\leq{}p+1$ the question is a bit more complicate because in general the norms of
$1-\theta\pi^\ell$ will not be contained in $U_{3,K}$, but a proper combination of norms of elements
of this form may be, and we should require it to be in $1+p^3V$. However for $\eta$ varying the
elements $N_{L/K}(1-\eta\pi)^p$ have norms which cover all classes modulo $U_3$, and consequently
each $N_{L/K}(1-\theta\pi^\ell)$ can be reduced into $U_{3,K}$ by multiplication by an suitable
$N_{L/K}(1-\eta\pi)^p$ for some $\eta$, and we should verify that all such reduction are actually in
$1+p^3V$. The condition for more complicated combinations will certainly also be ensured.

Since the map $\sfr{\kp_K^2}{\kp_K^4}\rightarrow\sfr{U_{2,K}}{U_{4,K}}$ induced by $x\mapsto1+x$ is
still an isomorphism we have that a proper adjustment of $1+\theta\pi^\ell$ (e.g. via the
Artin-Hasse exponential) has norm of the form $1+c_\ell(\theta)+d_\ell(\theta)$, where
$c_\ell(\theta)$ is obtained like in the case $p^2$ by with $f_i$ replaced by $f_{pi}$ and $T$ by
$T^p$, plus some modification for $\ell=2,3$. In other words depending on $2\leq\ell\leq{}p+1$ we
have that the remaining term which we call $g_\ell(Y)$ is
\begin{equation} \label{condb2}
  \begin{array}{cc}
   \left\{-f_{p^3}f_{p^2}Y^{p^2}+f_{p^2+p}Y^p\right\} -f_{p^2}f_pY^p + f_{p+1}Y  & [p+1], \\
   \left\{f_{p^2\ell}Y^{p^2}+ f_{p\ell}Y^p\right\} - f_{p^2}f_{p^2(\ell-1)}Y^{p^2} + f_{\ell}Y & [4,p-1], \\
   \left\{f_{3p^2}Y^{p^2}+f_{3p}Y^p\right\} + \frac{1}{3}f_{p^3}^3Y^{p^3} + \frac{1}{3}f_{p^2}^3Y^{p^2} 
                                    -f_{p^2}f_{2p^2}Y^{p^2} + f_{3}Y & [3], \\
   \left\{-\frac{1}{2}f_{p^3}^2Y^{p^3}+\left(f_{2p^2}
       -\frac{1}{2}f_{p^2}^2\right)Y^{p^2}+f_{2p}Y^p\right\} + f_{2}Y & [2], \\
  \end{array}
\end{equation}
where under braces are the terms which are not identically in $\kp_{K}^3$.  On the other hand
\[
   N(1+\eta\pi) = 1+f_{p^3}\eta^{p^3} + f_{p^2}\eta^{p^2} + f_{p}\eta^{p} \mod p^2V
\]
and consequently
\[
   N(1+\eta\pi)^p = 1+\left\{pf_{p^3}\eta^{p^3} + pf_{p^2}\eta^{p^2}\right\} + pf_{p}\eta^{p} \mod p^3V,
\]
with again under braces are the terms which are not identically in $\kp_{K}^3$.  Consequently let's
consider the polynomial
\[
   h(Z) = \{pf_{p^3}Z^{p^2} + pf_{p^2}Z^{p}\} + pf_{p}Z,
\]
we are looking for values of $Z=\phi_\ell(Y)$, that will be the lifting of some additive polynomials
in $Y$, such that $g_\ell(Y)-h(\phi_\ell(Y))\in\kp^3_K$, to impose the condition that it shall be in
$p^3V$ as well.

The connected additive polynomials $\overline{\sfr{g_\ell(Y)}{p^2}}$, which we denote by
$B_\ell(Y^p)$ replacing $Y^p$ by $Y$, are forced to have image contained $V$ which is the image of
$\overline{\sfr{h(Y)}{p^2}}=A(Y^{p})$, and the condition is that $B_\ell(Y)=A(D_\ell(Y))$ for some other
additive polynomial $D_\ell(Y)$ whose coefficients can be deduced easily.

In particular, being $A(Y)=F_{p^3}Y^p+F_{p^2}Y$ and $B_\ell(Y)$ the polynomials
\begin{equation} 
  \begin{array}{cc}
    -F_{p^2}F_{p^3}Y^p + G_{p^2+p}Y & [p+1], \\
  G_{p^2\ell}Y^p + G_{p\ell} Y & [3,\ p-1],\\
  -\frac{1}{2}F_{p^3}^2Y^{p^2}+\left(G_{2p^2}-\frac{1}{2}F_{p^2}^2\right)Y^p+G_{2p}Y & [2],
  \end{array}
\end{equation}
in view of Prop. \ref{prop_add} we can take as $D_\ell(Y)$ respectively the polynomials
\begin{equation} 
  \begin{array}{cl}
    \sfr{G_{p\ell}}{F_{p^2}}Y & [3,\ p+1], \\
   -\frac{1}{2}F_{p^3}^{\sfr{1}{p}}Y^{p^2}+\sfr{G_{2p}}{F_{p^2}}Y & [2].
  \end{array}
\end{equation}

Now, $B_\ell(Y^p)=A((D_\ell^{\sfr{1}{p}}(Y))^p)$ where $D_\ell^{\sfr{1}{p}}(Y)$ is $D_\ell(Y)$ with
the map $x\mapsto{}x^{\sfr{1}{p}}$ applied to the coefficients. Given the definitions of $A(Y)$ and
$B_\ell(Y)$ in terms of the $h(Y)$ and $g_\ell(Y)$, we have that we can take as $\phi_\ell(Y)$ any
lifting of $D_\ell^{\sfr{1}{p}}(Y)$ to $\cO_K[Y]$.

For $3\leq\ell\leq{}p+1$ let's take a $\rho\in\cO_K$ such that
$\bar\rho^p=\sfr{G_{p\ell}}{F_{p^2}}=\overline{\sfr{f_{p\ell}}{pf_{p^2}}}$, then
$D_\ell(Y)=\bar\rho{}^pY$ and we can take $\phi_\ell(Y)=\rho{}Y$, and the polynomials
$\overline{\frac{1}{p^3}(g_\ell(Y)-h(\phi_\ell(Y)))}$ should take values in $V$. Considering that
\[
   h(\phi_\ell(Y)) = \left\{pf_{p^3}\rho^{p^2}Y^{p^2} + pf_{p^2}\rho^{p}Y^{p}\right\} + pf_{p}\rho{}Y,
\]
depending on $\ell$ they are
  \begin{gather}
    \overline{\left(-\sfr{f_{p^3}f_{p^2}}{p^3}-\sfr{f_{p^3}\rho^{p^2}}{p^2}\right)}Y^{p^2} \\
      +\left[\overline{\left(\sfr{f_{p^2+p}}{p^3}-\sfr{f_{p^2}\rho^p}{p^2}\right)}-F_{p^2}G_p\right]Y^p + H_{p+1}Y
  \end{gather}
for $\ell=p+1$,
  \begin{gather}
    \left[\overline{\left(-\sfr{f_{p^2\ell}}{p^3}-\sfr{f_{p^3}\rho^{p^2}}{p^2}\right)}-F_{p^2}G_{p^2(\ell-1)}\right]Y^{p^2}\\
      +\overline{\left(\sfr{f_{p\ell}}{p^3}-\sfr{f_{p^2}\rho^p}{p^2}\right)}Y^p + H_\ell{} Y, 
  \end{gather}
for $4\leq\ell=p-1$, and
  \begin{gather}
    \frac{1}{3}F_{p^3}^3Y^{p^3} 
  + \left[\overline{\left(\sfr{f_{3p^2}}{p^3}-\sfr{f_{p^3}\rho^{p^2}}{p^2}\right)} + \frac{1}{3}F_{p^2}^3
    -F_{p^2}G_{2p^2} \right]Y^{p^2} \\
  + \overline{\left(\sfr{f_{3p}}{p^3}-\sfr{f_{p^2}\rho^p}{p^2}\right)}Y^p + H_{3}Y
  \end{gather}
for $\ell=3$.

For $\ell=2$ let's take $\rho,\tau\in\cO_K$ such that
$\bar\rho^p=\sfr{G_{p\ell}}{F_{p^2}}=\overline{\sfr{f_{p\ell}}{pf_{p^2}}}$ and
$\bar\tau^{p^2}=-\frac{1}{2}{F_{p^3}}=-\frac{1}{2}\overline{\sfr{f_{p^3}}{p^3}}$. Then
$D_\ell(Y)=\bar\tau^pY^p+\bar\rho^pY$ so that we can take $\phi_2(\ell)=\tau{}Y^p+\rho{}Y$, and we
have
\begin{gather*}
   h(\phi_2(Y)) = \left\{pf_{p^3}(\tau{}Y^p+\rho{}Y)^{p^2} 
    + pf_{p^2}(\tau{}Y^p+\rho{}Y)^{p}\right\} + pf_{p}(\tau{}Y^p+\rho{}Y)\\
  = pf_{p^3}\tau^{p^3}Y^{p^3}+pf_{p^3}\rho^{p^2}Y^{p^2} + pf_{p^3} \sum_{i=1}^{p-1} \binom{p^2}{ip}
  \tau^{ip}\rho^{(p-i)p}Y^{ip^2+(p-1)p} + \cO(p^4) \\
  + pf_{p^2}\tau^{p^2}Y^{p^2}+pf_{p^2}\rho^{p}Y^{p} + pf_{p^2} \sum_{i=1}^{p-1} \binom{p}{i}
  \tau^{i}\rho^{(p-i)}Y^{ip+(p-1)} \\
  + pf_{p}\tau{}Y^p+pf_{p}\rho{}Y.
\end{gather*}
Considering that $\frac{1}{p}\binom{p}{i} \equiv \frac{1}{p}\binom{p^2}{ip}\imod{p}$ and the terms
in the sums can be paired in elements that are $pf_{p^3}\binom{p^2}{ip}Z^p+pf_{p^2}\binom{p}{i}Z$
for $Z=\tau^{i}\rho^{(p-i)}Y^{ip+(p-1)}$ and hence in $p^3V$ for each $Z$, we have that up to some
element in $p^3V$ we can write $h(\phi_2(Y))$ as
\begin{equation*}
  pf_{p^3}\tau^{p^3}  Y^{p^3}
  +(pf_{p^3}\rho^{p^2} + pf_{p^2}\tau^{p^2})  Y^{p^2}
  +(pf_{p^2}\rho^{p} + pf_{p}\tau{})  Y^p
  + pf_{p}\rho{}Y.
\end{equation*}
Consequently up to some element of $V$ the polynomial
$\overline{\frac{1}{p^3}(g_\ell(Y)-h(\phi_2(Y)))}$ is the
\begin{equation*}
\begin{gathered}
\overline{\left(-\frac{1}{2}\sfr{f_{p^3}^2}{p^3}-\sfr{f_{p^3}\tau^{p^3}}{p^2}\right)}Y^{p^3}
 +\overline{\left(\sfr{f_{2p^2}}{p^3}-\frac{1}{2}\sfr{f_{p^2}^2}{p^3}
     -\sfr{f_{p^3}\rho^{p^2}}{p^2}-\sfr{f_{p^2}\tau^{p^2}}{p^2}\right)}Y^{p^2}\\
 +\left(\overline{\left(\sfr{f_{2p}}{p^3}-\sfr{f_{p^2}\rho^{p}}{p^2}\right)}-G_p\tau\right)Y^p +
  \left(H_{2}-G_p\rho \right)Y,
\end{gathered}
\end{equation*}
which is required to take values in $V$.

One last effort is required: for $\ell=1$ in the case that $1-\theta\pi$ has norm in $U_{2,K}$ (and
hence in $1+p^2V$), that is when $\theta$ is such that $A(\bar\theta^{p^2})=0$, we should also have
that taking $\eta$ such that $(1-\theta\pi)(1-\eta\pi)^{-p}$ has norm in $U_{3,K}$, than that norm
is required to be actually in $1+p^3V$.

Let $\theta=T$ be as required, the terms which disappear because $\ell=1$ are
\[
   \frac{1}{2}f_{p^2}^2T^{2p^2} + f_{p^3}f_{p^2}T^{p^3+p^2}+\frac{1}{2}f_{p^3}^2T^{2p^3}
    =\frac{1}{2}\left(f_{p^2}T^{p^2}+f_{p^3}T^{p^3}\right)^2,
\]
then
\[
   \frac{1}{3}f_{p^2}^3T^{3p^2} + f_{p^3}f_{p^2}^2T^{p^3+2p^2}+ f_{p^3}^2f_{p^2}T^{2p^3+p^2} +\frac{1}{3}f_{p^3}^3T^{3p^3}
    =\frac{1}{3}\left(f_{p^2}T^{p^2}+f_{p^3}T^{p^3}\right)^3,
\]
and the sums can be decomposed as sums of $(f_{p^2}T^{p^2}+f_{p^3}T^{p^3})f_{p^2j}T^{p^2j}$ and of
 $(f_{p^2}T^{p^2}+f_{p^3}T^{p^3})f_{pk}T^{pk}$, and in particular all such terms are in $\kp_K^4$
 considering the hypotheses on $T$.

Consequently such terms can be assumed to be present, and removing the extra terms we already
studied (or considering the norm of $E(\theta\pi)$) the remaining terms are
\[
  w(T) = f_{p^3}T^{p^3} + f_{p^2}T^{p^2}-f_{p^2}f_{p^3-p^2}T^{p^2}+f_pT^p+f_1{}T.
\]
Assume $\overline{\frac{1}{p^2}\left(f_{p^3}\theta^{p^3}+f_{p^3}\theta^{p^2}+f_p\theta^p\right)}$
can be written as $F_{p^3}\bar\alpha^{p^2}+F_{p^2}\bar\alpha^p$ for some $\bar\alpha$, then taking
any lift $\alpha$ of $\bar\alpha$ we can consider $w(\theta)-h(\alpha)$ which comes from a norm of
the required type, and should be in $p^3V$.

At last, we can state the

\begin{theo} \label{theox}
The Eisenstein polynomial $f(X)=X^{p^3}+f_1X^{p^3-1}+\dots+f_{p^3-1}X+f_{p^3}$ determines a Galois
extension of degree $p^3$ over $K$ if and only if
\begin{enumerate}
\item $v_p(f_{p^2})=1$ and $v_p(f_{p^2i})\geq{}2$ for $i\in\I{2,p-1}$,
\item $v_p(f_{pi})\geq{}2$ for all $i\in\I{1,p-1}$, $v_p(f_{p^2+p})=2$, and $v_p(f_{pi})\geq{}3$
  for all $i\in\I{p+1,p^2-1}$,
\item $v_p(f_{i})\geq{}3$ for all $i\in\I{1,p^2+p-1}$, $v_p(f_{p^2+p+1})=3$ and
  $v_p(f_{i})\geq{}4$ for all $i\in\I{p^2+p+2,p^3-1}$,
  \setcounter{enumi_saved}{\value{enumi}}
\end{enumerate}
putting $F_{p^2}=\overline{\sfr{f_{p^2}}{p}}$, $F_{p^3}=\overline{\sfr{f_{p^3}}{p}}$, and
$\overline{G_i=\sfr{f_i}{p^2}}$ for all $i$ in $p^2\I{2,p-1}$ or in $p\I{1,p+1}$ we have
\begin{enumerate}
\setcounter{enumi}{\value{enumi_saved}}
\item $-\sfr{F_{p^2}}{F_{p^3}}\in{}\kk_K^{p-1}$,
\item $G_{p(p+1)}^p=-F_{p^2}^{p+1}$,
\item $G_{p^2\ell}=F_{p^3}\left(\sfr{G_{\ell{}p}}{F_{p^2}}\right)^p$ for $\ell\in\I{3,p-1}$,
\item $G_{2p^2}=F_{p^3}\left(\sfr{G_{2p}}{F_{p^2}}\right)^p+\frac{1}{2}F_{p^2}\left(F_{p^2}-F_{p^3}^{\sfr{1}{p}}\right)$,
  \setcounter{enumi_saved}{\value{enumi}}
\end{enumerate}
if $\rho$ is such that $\bar\rho^{p(p-1)}=-\sfr{F_{p^2}}{F_{p^3}}$ we have (independently of $\rho$)
\begin{enumerate}
\setcounter{enumi}{\value{enumi_saved}}
\item
  $\overline{\frac{1}{p^2}\left(f_{p^3}\rho^{p^2}+f_{p^2}\rho^{p}+f_p\rho\right)}=F_{p^3}\alpha^p+F_{p^2}\alpha$
  for some $\alpha\in{}\kk_K$,
  \setcounter{enumi_saved}{\value{enumi}}
\end{enumerate}
putting $H_i=\overline{\sfr{f_i}{p^3}}$ for $i$ in $\I{1,p^2+p+1}$ or in $p\I{p+2,p^2-1}$ we have
\begin{enumerate}
\setcounter{enumi}{\value{enumi_saved}}
\item $-G_{p(\ell-p^2)}F_{p^3}=F_{p^3}(\sfr{H_{\ell}}{F_{p^2}})^p$ for $\ell\in\I{p^2+1,p^2+p+1}$,
\item $H_{p\ell}=F_{p^3}(\sfr{H_{\ell}}{F_{p^2}})^p$ for $\ell\in\I{2p+2,p^2-1}$,
\item $H_{p(2p+1)}-F_{p^2}G_{p(p+1)} =F_{p^3}(\sfr{H_{2p+1}}{F_{p^2}})^p+F_{p^2}(F_{p^3}F_{p^2})^{\sfr{1}{p}}$,
\item $H_{p\ell}-F_{p^2}G_{p(\ell-p)} = 
   F_{p^3}(\sfr{ H_{\ell} }{F_{p^2}})^p - F_{p^2}(F_{p^2(\ell-p)})^{\sfr{1}{p}}$ for $\ell\in\I{p+3,2p-1}$,
\item $H_{p(p+2)}-F_{p^2}G_{2p} = F_{p^3}(\sfr{ H_{p+2} }{F_{p^2}})^p +
  F_{p^2}(F_{p^2}^2-F_{2p^2})^{\sfr{1}{p}}$,
\setcounter{enumi_saved}{\value{enumi}}
\end{enumerate}
for each $\ell\in\I{3,p+1}$, let $\rho_\ell$ be such that
$\bar\rho_\ell^p=\sfr{G_{p\ell}}{F_{p^2}}$. Then
\begin{enumerate}
\setcounter{enumi}{\value{enumi_saved}}
\item putting
\[
    Q_{p+1} = \overline{\left(\sfr{f_{p^2+p}}{p^3}-\sfr{f_{p^2}\rho_{p+1}^p}{p^2}\right)}-F_{p^2}G_p,
   \qquad R_{p+1} = \overline{\left(-\sfr{f_{p^3}f_{p^2}}{p^3}-\sfr{f_{p^3}\rho_{p+1}^{p^2}}{p^2}\right)},
\]
we have $Q_{p+1} = F_{p^3}(\sfr{ H_{p+1} }{F_{p^2}})^p + F_{p^2}(\sfr{ R_{p+1} }{F_{p^3}})^{\sfr{1}{p}}$,
\item for each $4\leq\ell\leq{}p-1$ putting
\[
    Q_\ell = \overline{\left(\sfr{f_{p\ell}}{p^3}-\sfr{f_{p^2}\rho_\ell^p}{p^2}\right)}-F_{p^2}G_p,
   \qquad R_\ell = \overline{\left(-\sfr{f_{p^2\ell}}{p^3}-\sfr{f_{p^3}\rho_\ell^{p^2}}{p^2}\right)}-F_{p^2}G_{p^2(\ell-1)},
\]
we have $Q_\ell = F_{p^3}(\sfr{ H_{\ell} }{F_{p^2}})^p + F_{p^2}(\sfr{ R_\ell }{F_{p^3}})^{\sfr{1}{p}}$,
\item putting
\[
    Q_3 = \overline{\left(\sfr{f_{3p}}{p^3}-\sfr{f_{p^2}\rho_3^p}{p^2}\right)},
   \qquad R_3 = \overline{\left(\sfr{f_{3p^2}}{p^3}-\sfr{f_{p^3}\rho_3^{p^2}}{p^2}\right)} + \frac{1}{3}F_{p^2}^3 -F_{p^2}G_{2p^2}
\]
we have $\frac{1}{3}F_{p^2}(F_{p^3}^2)^{\sfr{1}{p}}+F_{p^3}(\sfr{Q_3}{F_{p^2}})^p=R_3+
F_{p^3}(\sfr{F_{p^3}}{F_{p^2}})^p(\sfr{H_3}{F_{p^2}})^{p^2}$,
\setcounter{enumi_saved}{\value{enumi}}
\end{enumerate}
let $\rho_2,\tau_2\in\cO_K$ such that $\bar\rho_2^p=\sfr{G_{p\ell}}{F_{p^2}}$ and
$\bar\tau_2^{p^2}=-\frac{1}{2}{F_{p^3}}$. Then
\begin{enumerate}
\setcounter{enumi}{\value{enumi_saved}}
\item putting
\begin{equation*}
\begin{gathered}
P_2 = H_{2}-G_p\bar\rho,\qquad Q_2 =
\overline{\left(\sfr{f_{2p}}{p^3}-\sfr{f_{p^2}\rho^{p}}{p^2}\right)}-G_p\bar\tau,\\
R_2 = \overline{\left(\sfr{f_{2p^2}}{p^3}-\frac{1}{2}\sfr{f_{p^2}^2}{p^3}
     -\sfr{f_{p^3}\rho^{p^2}}{p^2}-\sfr{f_{p^2}\tau^{p^2}}{p^2}\right)},\qquad
S_2 = \overline{\left(-\frac{1}{2}\sfr{f_{p^3}^2}{p^3}-\sfr{f_{p^3}\tau^{p^3}}{p^2}\right)}
\end{gathered}
\end{equation*}
we have $F_{p^2}(\sfr{S_2}{F_{p^3}})^{\sfr{1}{p}}+F_{p^3}(\sfr{Q_2}{F_{p^2}})^p=R_2+
F_{p^3}(\sfr{F_{p^3}}{F_{p^2}})^p(\sfr{P_2}{F_{p^2}})^{p^2}$,
\setcounter{enumi_saved}{\value{enumi}}
\end{enumerate}

if $\rho,\xi$ are such that $\bar\rho^{p^2(p-1)}=-\sfr{F_{p^2}}{F_{p^3}}$ and
\[
   \overline{\frac{1}{p^2}\left(f_{p^3}\rho^{p^3}+f_{p^2}\rho^{p^2}+f_p\rho^p\right)}
      =F_{p^3}\bar\xi^{p^2}+F_{p^2}\bar\xi^p,
\]
\begin{enumerate}
\setcounter{enumi}{\value{enumi_saved}}
\item we have that
\[
    \overline{\frac{1}{p^3}\left(f_{p^3}(\rho^{p^3}-\xi^{p^2})+f_{p^2}(\rho^{p^2}-\xi^{p})+f_p(\rho^p-\xi)
       -f_{p^2}f_{p^3-p^2}\rho^{p^2}+f_1\rho\right)}
\]
is also of the form $F_{p^3}\bar\omega^p+F_{p^2}\bar\omega$ for some $\bar\omega\in\kk_K$.
\end{enumerate}
\end{theo}

\section{Sums of roots of unity}
We finally prove the lemma about the $\Sigma_\lambda(\ell)$, it is actually much more than needed
but nevertheless is has a nice statement which could still be useful in similar circumstances:

\begin{lemma} \label{combinat}
Let $\lambda=(\lambda_1,\lambda_2,\dots,\lambda_r)$ be a partition, then
\[
   \Sigma_\lambda(\ell) = \sum_{\lambda=\sqcup_{j\in{}J}\lambda^{(j)}} 
    \ell^{\#J} \cdot \prod_{j\in J} (-1)^{\#\lambda^{(j)}-1} (\#\lambda^{(j)}-1)!
\]
where the sum is over all the partitions $\lambda=\bigsqcup_{j\in{}J}\lambda^{(j)}$ (as set) such
that for each $j\in{}J$ the sum $|\lambda^{(j)}|$ of the elements in $\lambda^{(j)}$ is multiple of
$\ell$ and $\#\lambda^{(j)}$ is the cardinality of the subset $\lambda^{(j)}$.
\end{lemma}
\begin{proof}
  Let $A_{(i,j)}$ be the sets (indexed by the pairs $(i,j)$) of indices $(\iota_1,\dots,\iota_r)$
  such that $\iota_i=\iota_j$, let $A_0$ be the set of all possible indices, and for
  $A\subseteq{}A_0$ denote by $\Sigma(A)$ the sum over all the indices in $A$. By
  inclusion-exclusion we have that
\begin{align*}
  \Sigma_\lambda(\ell) &= \Sigma(A_0)-\Sigma(\cup_{(i,j)} A_{(i,j)}) \\
    &= \Sigma(A_0)-\sum_{(i,j)}\Sigma(A_{(i,j)}) + \sum_{(i,j)\neq (i',j')}\Sigma(A_{(i,j)}\cap A_{(i',j')}) - \dots
\end{align*}
Now let $A$ be the intersection of all the sets $A_{(i_k,j_k)}$ for a collection of pairs
$P=\{(i_1,j_1),\dots,(i_s,j_s)\}_{s\in{}S}$, if we consider the graph with $R=\{1,\dots,r\}$ as
vertices and the $(i_k,j_k)$ as edges we have that if we split $R$ in connected components
$R=\sqcup_{t\in{T}}R_t$ then the allowed indices $\iota$ are those constant on each $R_t$, and
calling $\iota_t$ the value taken on $R_t$ the sum $\Sigma(A)$ becomes
\[
  \Sigma(A) = \prod_{t\in T} \sum_{\iota_t=0}^{\ell-1}\zeta_\ell^{\iota_t(\sum_{r\in{}R_t}\lambda_r)},
\]
and this sum is $\ell^{\#T}$ when all the $\sum_{r\in{}R_t}\lambda_r$ are multiple of $\ell$ and $0$
if not. Note that $\Sigma(A)$ appears with sign equal to $(-1)^{\#S}$ in the inclusion-exclusion, so
for each partition of $R$ in sets $R_t$ such that the sum of $\lambda_r$ for $r\in{}R_t$ is multiple
of $\ell$ we have that to consider the all graphs with set of vertices $R$ and such that each $R_t$
is a connected component, and count the number of graphs with an even number of edges minus those
with a odd number of edges. Now the total difference is the product of the differences over all the
connected components, so we have
\[
   \Sigma_\lambda(\ell) = \sum_{\lambda=\sqcup_{j\in{}J}\lambda^{(j)}} \ell^{\#J} \cdot \prod K_{\#\lambda^{(j)}}
\]
where for each $i$ we denote by $K_i$ the difference of the number of connected graphs on $i$
vertices having an even and odd number of edges.

The difference of the number of connected graphs $K_i$ on $i$ vertices with an even or odd number of
vertices can be computed fixing an edge, and considering the graphs obtained adding or removing that
edge. Those such that with or without it are connected come in pairs with an even and odd number of
edges, the other graphs are obtained connecting two other connected graphs on $j$ and $i-j$
vertices. In particular choosing $j-1$ vertices to make one component with the first vertex of our
distinguished edges we obtain
\[
 K_{i+2} = - \sum_{j}^{i} \binom{i}{j} K_{i-j+1}K_{j+1}
\]
for $i\geq0$, and $K_1=1$. Calling $G(X)$ the exponential generating function
$\sum_{i=0}^\infty\frac{K_{i+1}}{i!}X^i$ we obtain that
\[
  \frac{d}{dX}G(X) = -G(X)^2
\]
with the additional condition that $K_1=1$, and this equation is clearly satisfied by $\sfr{1}{1+X}$
which can be the only solution. Consequently $K_{i+1}=(-1)^i\cdot{}i!$ and the lemma is proved.
\end{proof}

\bibliographystyle{amsalpha}
\bibliography{biblio}

\end{document}